\documentclass[12pt]{amsart}

\pdfoutput=1

\usepackage[text={420pt,660pt},centering]{geometry}

\usepackage{xcolor}
\usepackage{esint,amssymb}
\usepackage{graphicx}
\usepackage{MnSymbol}
\usepackage{mathtools}
\usepackage[colorlinks=true, pdfstartview=FitV, linkcolor=blue, citecolor=blue, urlcolor=blue,pagebackref=false]{hyperref}
\usepackage{microtype}

\usepackage{bm}
\usepackage{scalerel} 
\usepackage{dsfont}
\usepackage[font={footnotesize}]{caption}

\usepackage{lipsum}
\definecolor{darkgreen}{rgb}{0,0.5,0}
\definecolor{darkblue}{rgb}{0,0,0.7}
\definecolor{darkred}{rgb}{0.9,0.1,0.1}

 \usepackage{comment}
 
 \includecomment{comment}
 
 \specialcomment{supplementary}
    {\colorlet{savedcolor}{.} \color{blue} \begingroup \ttfamily \bigskip \smallskip  \noindent \underline{Supplementary details:} \newline \newline \footnotesize }{\endgroup \smallskip \color{savedcolor}}
\excludecomment{supplementary}

\newtheorem{proposition}{Proposition}
\newtheorem{theorem}[proposition]{Theorem}
\newtheorem{lemma}[proposition]{Lemma}

\theoremstyle{definition}

\newtheorem{definition}[proposition]{Definition}

\newcommand{\cref}[1]{Corollary~\ref{c.#1}}

\numberwithin{equation}{section}
\numberwithin{proposition}{section}

\newcommand{\A}{\mathcal{A}}
\newcommand{\Ahom}{\overline{\A}}

\newcommand{\Z}{\mathbb{Z}}
\newcommand{\N}{\mathbb{N}}
\newcommand{\R}{\mathbb{R}}

\newcommand{\E}{\mathbb{E}}
\renewcommand{\P}{\mathbb{P}}
\newcommand{\F}{\mathcal{F}}
\newcommand{\Zd}{\mathbb{Z}^d}
\newcommand{\Rd}{{\mathbb{R}^d}}

\newcommand{\ep}{\varepsilon}

\renewcommand{\a}{\mathbf{a}}

\newcommand{\g}{\mathbf{g}}

\renewcommand{\subset}{\subseteq}

\renewcommand{\a}{\mathbf{a}}
\renewcommand{\b}{\mathbf{b}}

\newcommand{\ahom}{{\overbracket[1pt][-1pt]{\a}}}  

\renewcommand{\subset}{\subseteq}

\newcommand{\cu}{{\scaleobj{1.2	}{\square}}}

\renewcommand{\fint}{\strokedint}

\DeclareMathOperator{\dist}{dist}

\DeclareMathOperator*{\esssup}{ess\,sup}

\DeclareMathOperator{\tr}{tr}

\DeclareMathOperator{\spn}{span}

\DeclareMathOperator{\mindiam}{md}

\newcommand{\X}{\mathcal{X}}

\renewcommand{\bar}{\overline}

\renewcommand{\tilde}{\widetilde}

\newcommand{\indc}{\mathds{1}}

\renewcommand{\O}{\mathcal{O}}

\renewcommand{\hat}{\widehat}

\begin{document}

\title[Homogenization of the obstacle problem]{Quantitative homogenization for the obstacle problem and its free boundary}

\begin{abstract}
In this manuscript we prove quantitative homogenization results for the obstacle problem with bounded measurable coefficients.
As a consequence,  large-scale regularity results  both for the solution and the free boundary for the 
 heterogeneous  obstacle problem are derived.

\end{abstract}

\author[G. Aleksanyan]{Gohar Aleksanyan}
\address[G. Aleksanyan]{Department of Mathematics and Statistics, University of Helsinki}
\email{gohar.aleksanyan@helsinki.fi}

\author[T. Kuusi]{Tuomo Kuusi}
\address[T. Kuusi]{Department of Mathematics and Statistics, University of Helsinki}
 \email{tuomo.kuusi@helsinki.fi}

\date{\today}

\maketitle
\setcounter{tocdepth}{1}
\tableofcontents

\section{Introduction}


%

In this manuscript we investigate the stochastic homogenization of the obstacle problem.  Let~$U \subset \mathbb{R}^d$ be a given Lipschitz domain. Let ~$\varphi \in H^{1}(U)$ be the given obstacle function and let~$g \in H^{1}(U)$ be boundary values, and assume that~$g \geq \varphi$ on~$\partial U$. Assuming that the measurable coefficient field~$\a$ is uniformly elliptic, i.e.,~$I_d \leq \a(x) \leq \Lambda I_d$ for almost every $x \in \R^d$, we consider the following energy 
\begin{equation} \label{J.heter}
 J_\ep[u]= \frac12 \int_U  \a\big( \tfrac{x}{\ep} \big)\nabla u(x) \cdot \nabla u(x) dx,
\end{equation}
and the minimization problem 
\begin{align}  \label{J.heter.min}
\min_{u \in \mathcal{K}}  J_\ep[u]  , \quad  \mathcal{K} := \big\{  u \in g + H_0^1(U) \, : \,  u \geq \varphi   \big\} \,. 
\end{align}
Notice that the convex set~$\mathcal{K}$  is not empty, because we assume that~$g \geq \varphi$ on~$\partial U$. 

\smallskip

The existence  and uniqueness of the minimizer can be shown  via direct methods in the calculus of variation, see for example,~\cite{Fribook}.
Via  variational methods it is easy to verify that the minimizer~$u^\ep$ satisfies the following system of inequalities in a weak sense: 
\begin{align} \label{e.sol}
\left\{
\begin{aligned}
& -\nabla \cdot \a^\ep \nabla u^\ep   \geq  0  
& &
\mbox{ in } U
\\
&
-\nabla \cdot \a^\ep \nabla u^\ep   = 0  
& &
\mbox{ in } \{ u^\ep > \varphi \} \cap U
\\
&  u^\ep \geq  \varphi
&& 
\mbox{ in } U
\\
&  u^\ep = g 
&& 
\mbox{ on } \partial U
.
\end{aligned}
\right.
\end{align}
Here, and in what follows, we denote~$\a^\ep : = \a\big( \tfrac{\cdot}{\ep} \big)$ for small $\ep \in (0,1)$, and assume that~$\a$ is symmetric and satisfies the following uniform ellipticity condition: There exists a constant $\Lambda \in [1,\infty)$ such that, for almost every $x \in \R^d$ and every $\xi \in \R^d$, 
\begin{align} \label{e.ellipticity}
\a(x) \xi  \cdot \xi \geq | \xi|^2 \quad \mbox{and} \quad |\a(x)| \leq \Lambda .
\end{align}
The set~$\varOmega(u^\ep):= U \cap \{ u^\ep>\varphi\}$ is called the \emph{non-coincidence set}, and 
$\varGamma(u^\ep):=\partial \varOmega( u^\ep) \cap U$ is called the  \emph{the free boundary}. The \emph{coincidence} or \emph{contact set} is denoted by $\varLambda(u^\ep):=\{ u^\ep \equiv \varphi \}$.
By~\eqref{e.sol}, the minimizer~$u^\ep$  is a weak supersolution over the whole domain~$U$  and~$u^\ep$ is~$\a^\ep$-harmonic in the non-coincidence set. Observe that the free boundary~$\varGamma(u^\ep)$ depends on the solution~$u^\ep$ and is not known a priori. Moreover, we see that~$\nabla \cdot(\a^\ep \nabla u^\ep) =  \nabla \cdot(\a^\ep \nabla \varphi)$ in the interior of the set~$\{u^\ep =\varphi \}$, while 
$\nabla \cdot(\a^\ep \nabla u^\ep) = 0$ in the non-coincidence set. Hence,  in general, the flux~$\a^\ep \nabla u^\ep$ is discontinuous along the free boundary.

\smallskip

The obstacle problem has numerous applications, from option pricing in financial mathematics to tumor growth problems in biology.  Free  boundary problems of similar nature appear in mathematical modeling of various problems in physics and life sciences. Our contribution is to study these type of problems in heterogeneous, possibly random, medium.

\smallskip

When investigating free boundary problems, we ask two main questions. First, what is the optimal regularity of the solution? Second, what can we say about the regularity of the free boundary? In the context of homogenization, we additionally ask the question whether we can approximate the solution~$u^\ep$ with a solution to the corresponding homogenized obstacle problem as~$\ep \to 0$: Minimize
\begin{equation} \label{J.homog}
J[u]=  \frac 12 \int_U  \ahom \nabla u(x) \cdot \nabla u(x) dx,
\end{equation}
over~$u \in \mathcal{K}$. The symmetric matrix~$\ahom$ is the \emph{homogenized matrix}, which can be identified under suitable deterministic or stochastic \emph{ergodicity} assumption. The exact assumptions on the coefficient field will be  discussed in Subsection~\ref{s.homogenization}.  The subsequent question is if and when we can prove homogenization of the corresponding free boundaries?

\smallskip

Our goal is to give a rather general condition on the coefficient  matrix $ \a$, which allows to treat both \emph{qualitative} and \emph{quantitative} homogenization at the same time. Indeed, we show that under a suitable \emph{sublinearity condition} for the correctors and the flux correctors, described in Section~\ref{s.homogenization}, we obtain homogenization estimates. More precisely, for $m = \lceil - \log_3 \ep  \rceil$ and $\cu_m = ( -\tfrac12 3^m , \tfrac12 3^m )^d$, we let $\phi_{m,e} \in H^1(\cu_m)$ solve the equation  
\begin{align} 
\label{e.corrector.intro}
 \nabla \cdot \a ( e+ \nabla \phi_{m,e}) = 0  \quad \mbox{in } \cu_m,
\end{align}
and then find a $3^m \Z^d$-periodic skew-symmetric ($\mathbf{S}_{ij} = -\mathbf{S}_{ji}$) tensor $\mathbf{S}$ solving
\begin{align}  \label{e.fluxcorrector.intro}
\Delta \mathbf{S}_{m,ij}^{e} = \partial_{x_j} \big( \a ( e+ \nabla \phi_{m,e})\big)_i - \partial_{x_i} \big( \a ( e+ \nabla \phi_{m,e})\big)_j, \quad
\mathbf{S}_{m,ij}^{e} \in H_{\mathrm{per}}^1(\cu_m) ,
\end{align}
and coarsened homogenized coefficients via 
\begin{align} \label{e.abar.intro}
\ahom_m e  = \fint_{\cu_m} \a ( e+ \nabla \phi_{m,e}) .
\end{align}
Our assumption then reads as that there exists a homogenized matrix $\ahom$ and a monotone increasing function $t\mapsto \mathcal{E}(t)$ such that the quantity 
\begin{equation} \label{e.error.intro}
| \ahom_m - \ahom | + 3^{-m} \sum_{k=1}^d \left( \| | \mathbf{S}_m^{(e_k)} |+  |\phi_{m,e_k}|  \|_{\underline{L}^2(\cu_m)} \right)
\leq 
\mathcal{E}(\ep) 
, \quad m := \lceil - \log_3 \ep\rceil 
\end{equation}
is small. Our homogenization estimates are all given in terms of $ \mathcal{E}(\ep)$. Notice that~$ \mathcal{E}(\ep)$ is centered at the origin and it is in fact a function of the coefficient field~$\a$ in~$\cu_m$. We suppress this dependence from the notation. However, one should keep in mind that $ \mathcal{E}(\ep) = \mathcal{E}(\ep , \a |_{\cu_m} )$ and we can thus extend the definition of~$\mathcal{E}(\ep)$ for any given point $x \in \R^d$ using translation by setting 
\begin{align}  \label{e.error.translation}
\mathcal{E}(\ep,x) = \mathcal{E}(\ep , \a(x+\cdot) |_{\cu_m} ).
\end{align}
For example, in the random case  a typical form would be that $\mathcal{E}(\ep,x) = \X(x) \ep^\alpha$ for a stationary random field $x \mapsto \X(x)$ with certain stochastic integrability.  Our goal in this paper, in a nutshell, is to show that \emph{smallness of $\mathcal{E}(\ep)$ implies large-scale regularity in the obstacle problem.} We would like to point out here, however, that we are not aiming at \emph{optimal rates} of homogenization in the obstacle problem, since this would call much more careful analysis. We will not discuss this point any further in this manuscript. 

\smallskip

For the regularity results, we need a stronger, more quantified control of the error. Indeed, for a given parameter~$\nu \in (0,1)$, 
we assume that, for every $s,t \in (0,\infty)$ such that $s \leq t$, we have that 
\begin{align}  \label{e.algebraic}
s^{-\nu}  \mathcal{E} (s) \leq t^{-\nu}  \mathcal{E}(t) 
\, .
\end{align}
In particular, this is satisfied with $\mathcal{E}(\ep)  = \X \ep^\nu$, where $\X$ is a universal constant possibly depending on the realization of the coefficient field $\a$. Under the assumption~\eqref{e.algebraic}, we also define the so called  \emph{minimal scale}: for every $\alpha,\sigma \in (0,1]$ and~$x \in \R^d$, set
\begin{align}  \label{e.minimalscale.intro}
R_{\sigma,\alpha}(x) := \inf \bigl\{ r \in (0,\infty) \, : \, 
\bigl(  \mathcal{E}\bigl( \tfrac{\ep}{r} , x \bigr)\bigr)^\alpha <  \sigma 
\bigr\},
\qquad 
R_{\sigma,\alpha} = R_{\sigma,\alpha}(0)
\,.
\end{align}
This object is fundamental in the regularity theory, especially in the context of the stochastic homogenization, originating from the paper of Armstrong and Smart~\cite{AS} in the random setting, see also~\cite{AKMBook,GNO,GO6}. The original ideas on large-scale regularity go back to seminal papers of Avellaneda-Lin~\cite{AL1,AL2,AL3}. 

\smallskip
Now 
let us proceed to more detailed discussion of the problem under investigation. Let $ V^\ep:=u^\ep -\varphi \geq 0$, then we derive from~\eqref{e.sol} that $ V^\ep \geq 0 $ satisfies the following variational inequality in a weak sense
\begin{equation*}
\nabla \cdot ( \a^\ep \nabla V^\ep) = - \nabla \cdot ( \a^\ep \nabla \varphi) \chi_{\{V>0\}},
\end{equation*}
with corresponding boundary conditions. In order to be able to control the growth  of the solution~$V^\ep$ at free  boundary points, we need some
control over~$\nabla \cdot ( \a^\ep \nabla \varphi)$, which oscillates rapidly  as~$\ep \to 0+$ and is, in general, merely a distribution in~$H^{-1}(U)$. 
To remedy this, we notice that for a given obstacle $\varphi$ and boundary values $g \in H^1(U)$, we obtain by integration by parts that the energy in~\eqref{J.heter} can be rewritten as 
\begin{align}  \label{e.energydifference}
 J_\ep[u] 
 & 
 = 
\int_U \left(  \frac 12 \a^\ep \nabla (u-\varphi) \cdot  \nabla (u-\varphi ) + (u-\varphi ) (-\nabla \cdot \a^\ep \nabla \varphi )\right) 
\\ & \notag
\qquad 
+
\int_U \left(  
(g - \varphi)(-\nabla \cdot \a^\ep \nabla \varphi) 
+  
\a^\ep \nabla \varphi  \cdot  \nabla (g - \varphi)
+
\frac12  \a^\ep \nabla \varphi  \cdot \nabla \varphi
\right) 
 \,.
\end{align}
The second integral on the right is null-Lagrangian for fixed obstacle and boundary data, so letting $ v= u -\varphi \geq 0$,  in fact we obtain that  $ v $ is a nonnegative minimizer of the first integral on the right, with corresponding boundary data $g-\varphi$.  This leads us to consider the so-called \emph{normalized obstacle problem}: For given~$\varphi \in H^{1}(U)$, define~$\varphi^\ep  \in \varphi + H_0^1(U)$ to be the solution of
\begin{equation}
\label{e.obstacles.intro}
-\nabla \cdot \left(\a^\ep  \nabla \varphi^\varepsilon \right) =
-\nabla \cdot  \left(\ahom \nabla \varphi \right):=f \quad \mbox{in } U \,. 
\end{equation}
Letting $ v^\ep=u^\ep -\varphi^\ep \geq 0 $ and $ v =u-\varphi \geq 0$ we derive the following Euler-Lagrange equations from the minimization problem, now with the normalized right-hand side:
\begin{align}
\label{e.normheter.intro}
\nabla \cdot (\a^\ep \nabla v^\ep)=f \chi_{\{ v^\ep>0\} }
\quad \textrm{ and } \quad 
\nabla \cdot (\ahom \nabla v)= f \chi_{\{v>0\}}
\,.
\end{align}
In order to obtain regularity of the free boundary, and large-scale regularity, we  need to  impose more assumptions on~$f$, such as uniform positivity and H\"older continuity. As a matter of fact, in Section~\ref{s.fb} we simplify even further and take $f$ to be a positive constant $\lambda$. There are counterexamples, showing that if we allow~$f=0$ or~$f$ to be merely continuous, even when $ \a$ is a constant matrix, we cannot expect regularity of the free boundary, see~\cite{B01} for further discussion. The normalization described above, and the assumptions on~$f$ are widely used in the context of free boundary problems. In fact in most of the literature  $ f \equiv 1$, since the more general situation can be addressed via  stability properties of the free boundary, see \cite{B01} for details.

\smallskip

In view of~\eqref{e.energydifference}, studying the normalized obstacle problem described above is the same as if we study the minimizer of~\eqref{J.heter.min} over~$\mathcal{K}_\ep := \big\{  u \in g + H_0^1(U) \, : \,  u \geq \varphi^\ep   \big\}$. In other words, we intrinsically add the wild oscillations to the obstacle and only after that study the obstacle problem. The solution to the original obstacle problem~\eqref{J.heter.min}, denoted now by~$u^\ep_\varphi$,  is close to the solution of the obstacle problem with the new obstacle~$\varphi^\ep$.
Indeed, comparison principle yields that 
\begin{align}  \label{e.uep.vs.vep}
\left\|  u^\ep_\varphi - u^\ep \right\|_{L^\infty } 
\leq 
\left\|  \varphi^\ep - \varphi \right\|_{L^\infty } 
= O(\ep),
\end{align}
where the  smallness of $ \left\|  \varphi^\ep - \varphi \right\|_{L^\infty }$ follows from  the homogenization theory, see Lemma~\ref{l.wepandw.close} and Lemma \ref{l.comparison} for details. Furthermore,  the above estimate suggests that the homogenization of  solutions of the  normalized  obstacle problem implies homogenization of solutions to the original obstacle problem~\eqref{J.heter.min}, with a fixed obstacle. By the triangle inequality we have that 
\begin{align}
\notag
\left\|  u^\ep_\varphi - u \right\|_{L^\infty } 
&
 \leq 
 \left\|  v^\ep  - v  \right\|_{L^\infty } +2 \left\|  \varphi^\ep  -\varphi  \right\|_{L^\infty } ,
\end{align}
where $v^\ep$ and $v$ solve equations in~\eqref{e.normheter.intro}. 
Consequently, it suffices to investigate the homogenization of solutions of the normalized obstacle problem. Moreover, observe that the homogenization of the free boundary does not hold when the obstacle is fixed due to wild oscillations of $\a^\ep$ for small $\ep$ even if $\a$ is smooth. This can be seen already in 1D-examples as demonstrated in Appendix~\ref{s.appendix}.
 
 \smallskip

Let us briefly discuss the illuminating examples of Appendix~\ref{s.appendix}. We actually give two examples in the appendix. In Example 2 we show that even if we have a fixed strictly concave smooth obstacle~$\varphi$, then the we may not have homogenization of the free boundary due to rapid oscillations of~$\nabla \cdot (\a^\ep \nabla \varphi)~$as~$\ep \to 0+$, see Figure~\ref{f.oned.example} for numerical illustration. On the contrary, Example 1 shows that the homogenization of the solution and of the free boundary holds for the normalized obstacle problem. Letting~$v^ \ep \geq 0$ and~$v \geq 0$, respectively,  solve 
\begin{align*}
\nabla \cdot (\a^\ep \nabla v^\ep)= \chi_{\{ v^\ep>0\} }
\quad
\textrm{and}
\quad
\nabla \cdot (\ahom \nabla v)= \chi_{\{v>0\}}
\end{align*}
in the interval~$(0, 4)$, and letting~$\alpha^\ep$ and~$\alpha$ be the free boundary points for~$v^\ep$ and~$v$, respectively, we show that then~$v^\ep \to v$ and~$\alpha^\ep \to \alpha$.

\subsection{The  normalized obstacle problem }

Let $ f \in L^q$, $q >d/2$, $ f >0$ a.e., and  consider the minimizer of the  following convex functional
\begin{equation} \label{epsilonfunctional}
J_\varepsilon[v]:= \int_U \frac {1}{2}\nabla v(x) \cdot \a \left(\tfrac{x}{\varepsilon} \right) \nabla v(x) +f(x ) v(x)dx,
\end{equation}
over~$v \in H^{1}(U)$, satisfying~$v \geq 0$ in~$U$,~$ v =g >0$ on~$\partial U$.
There is a unique minimizer~$u^\varepsilon$, which we call  the \emph{heterogeneous solution} to the  normalized obstacle problem.  Then 
$u^\varepsilon$ is a weak solution to
\begin{equation}
\nabla \cdot \left( \a^\ep \nabla u^\varepsilon \right)=f \chi_{\{ u^\varepsilon > 0\}}.
\end{equation}
Let ~$u$ be the minimizer of the corresponding convex functional with the homogenized matrix~$\ahom$;
\begin{equation}
J[v]:= \int_U \frac{1}{2} \nabla v(x) \cdot \ahom \nabla v(x) +f(x) v(x)dx,
\end{equation}
over~$v \in H^{1}(U)$, satisfying~$v \geq 0$ in~$U$,~$ v=g >0$ on~$\partial U$.
Then
\begin{equation}
\nabla \cdot \left( \ahom  \nabla u \right)= f \chi_{\{ u >0\}}.
\end{equation}
The function~$u$ is called the \emph{homogenized solution} to the  normalized obstacle problem.

The corresponding free boundaries for $ u^\ep $ and $ u$ are defined as follows 
\begin{align}
\varGamma(u^\ep) =\partial \{ u^\ep >0\} \cap U 
\quad \mbox{and} \quad
 \varGamma(u) = \partial  \{ u >0\} \cap U.
\end{align}

\smallskip

Concerning the regularity of the minimizer $ u^\ep$, for merely measurable coefficients, the solution is H\"older regular by De Giorgi-Nash-Moser type arguments, see for example ~\cite{GTbook}.
 If coefficients are~$C^{0,\alpha}$ for some~$\alpha \in (0,1]$, then the minimizer is~$C^{1,\alpha}$-regular. 
When the coefficient matrix~$\a$ is constant (or sufficiently smooth), both the regularity of the solution and of the free boundary are well understood. The optimal  regularity of solutions  is due to Frehse, \cite{Frehse}.   When  $f $ is uniformly positive and H\"older continuous, the free boundary regularity is fully investigated by Caffarelli, \cite{Caf77, Caf98}, see also  some recent contributions such as ~\cite{FOS, FS}.
The more general case, when  $f $ is  Dini-continuous, the regularity of the free boundary is derived in see~\cite{ B01}.

\smallskip 

The obstacle problem with sufficiently regular variable coefficients have been studied  using energy methods, see for example~\cite{MP, FGS}. However, there are more challenges appearing when the coefficients are merely measurable. Blank, jointly with collaborators, studied  free boundary problems with rough coefficients,~\cite{BK,BZ1,BZ2}. 
In \cite{BZ1}, the authors  prove quadratic decay  bounds and non-degeneracy at free boundary points assuming that $ \a_{ij}$ and $ f $ are bounded measurable and  uniformly positive, see Lemma \ref{l.BlankHao}. In particular the authors show that the free boundary has  zero Lebesgue measure. In \cite{BK,BZ1,BZ2} the authors discuss the regularity of the free boundary, additionally assuming that $ a_{ij}, f \in VMO$.

\subsection{Summary of main results}

Now let us refer to the results obtained in the current manuscript. We investigate the closeness of $ u^\ep$ and $ u$,  the corresponding free boundaries $\varGamma(u^\ep)$ and $\varGamma(u)$, and derive large-scale regularity results for~$u^\ep$ and~$ \varGamma(u^\ep)$. 

Following \cite{AKMBook}, we denote by $ \overline{\mathcal{A}}_k$ and $ {\mathcal{A}}_k$  respectively the spaces of $\ahom$-harmonic and $\a$-harmonic functions, with growth strictly less than polynomials of degree $k+1$. More rigorously,
\begin{align}  \label{e.Ak.bar}
{\overline{\mathcal{A}}}_k
:=
\bigl\{ u \in H^1_{loc}( \R^d)
\, : \,
-\nabla \cdot (\ahom \nabla u )=0, ~ \lim_{r\to \infty} r^{-k-1} \|u \| _{{\underline{L}^2(B_{r})}}=0
\bigr
\},
\end{align}
and
\begin{align}  \label{e.Ak}
{\mathcal{A}}_k
:=
\bigl\{ u \in H^1_{loc}( \R^d)
\, : \,
-\nabla \cdot (\a \nabla u )=0, ~ \lim_{r\to \infty} r^{-k-1} \|u \| _{{\underline{L}^2(B_{r})}}=0
\bigr
\}.
\end{align}
By the Liouville theorem, $ \overline{\mathcal{A}}_1$ is the space of affine functions.
Under the assumption~\eqref{e.algebraic}, it turns out that ${\mathcal{A}}_1$ can be identified with the limits of $\phi_{m,e}$ in~\eqref{e.corrector.intro} as $m \to \infty$. The dimension of~${\mathcal{A}}_1$ is $d+1$ and it can be thought as a space of ``heterogenous affine functions". This space gives the right concept to measure $C^{1,1}$-regularity of solutions of the obstacle problem.
These type of large-scale regularity results were originally developed in the periodic setting by Avellaneda and Lin in a series of papers~\cite{AL1,AL2,AL3}. As a matter
of fact, a version of the following theorem has been commented in the periodic non-divergence case already in~\cite{AL2}. Our contribution here is to prove the corresponding result also in the stochastic setting. 

\begin{theorem} \label{t.C11}
Let $\nu \in (0,1]$ and assume that $\mathcal{E}$ satisfies~\eqref{e.algebraic}. There exist constants $\alpha,\sigma \in (0,1]$ and $C< \infty$, all depending only on $(\nu,d,\Lambda)$, such that the following claim holds. If $R_{\sigma,\alpha}$ defined in~\eqref{e.minimalscale.intro} satisfies~$R_{\sigma,\alpha}\leq 1$,~$f \in C^{0,\beta}(B_{1})$ and~$u^\ep$ is a weak solution of 
\begin{equation}
\nabla \cdot ( \a^\ep \nabla u^\ep) = f \chi_{\{u^\ep>0\}} \textrm{ in } B_1,
\end{equation}
then there exists $\psi \in \A_1$ such that, for every $r\in  [ R_{\sigma,\alpha},1]$, we have
\begin{equation} \label{e.C11}
\| u^\ep - \psi^\ep  \|_{L^\infty(B_{r})} 
\leq 
Cr^2   \big(\| u^\ep \|_{\underline{L}^2(B_{1})}
+
C \| f \|_{C^{0,
\nu}(B_1)} \big).
\end{equation}
\end{theorem}

\smallskip
Now let us discuss our results on the regularity of the free boundary $\varGamma(u^\ep)$. 
Let $ \lambda >0$, and define  the set of homogeneous half-space solutions to the homogenized obstacle problem $ \nabla \cdot (\ahom \nabla u)=\lambda \chi_{\{u>0\}}$, as follows
\begin{equation}\label{e.homog.blow-up}
\mathbb{H}_\lambda^0:=\left\{  h(x)= \frac{\lambda}{2} \frac{(x \cdot e)_+^2}{e\cdot \ahom e} \, : \,  e \in \partial B_1  \right\}.
\end{equation}
The following theorem is the second main result of this article. 
 \begin{theorem} \label{t.fb.regularity}
Let $\ep\in (0,1)$, $\beta,\vartheta \in (0,\tfrac12)$ and $\lambda > 0$. Let $\nu \in (0,1)$ be such that the condition~\eqref{e.algebraic} holds. There exist constants $\alpha(d,\Lambda)\in (0,1)$, $C(\nu,\beta,\vartheta,\lambda,d,\Lambda)  < \infty$, and $\sigma, r_0\in (0,1)$, $\gamma  \in (0,r_0]$, all depending only on $(\beta,\vartheta,\lambda,d,\Lambda)$, such that the following claim is valid. Let~$u^\ep$ solve the normalized obstacle problem 
\begin{align}  \label{e.thm.sol}
\nabla \cdot (\a^\ep \nabla u^\ep)= \lambda \chi_{\{ u^\ep>0\} }  \quad \mbox{in } B_1, 
\end{align}
 and assume that $0 \in \varGamma(u^\ep)$. Suppose that 
 \begin{align}  \label{e.thm.regular}
\inf_{r \in [\gamma r_0,\gamma^{-1} r_0]}  \frac{|  \varLambda(u^\ep)  \cap B_r|}{|B_r|} \geq \vartheta >0
\end{align}
and that~$R_{\sigma,\alpha} \leq \gamma r_0$, where~$R_{\sigma,\alpha}$ is defined in~\eqref{e.minimalscale.intro}. Then, for every $r \in [R_{\sigma,\alpha},\frac12)$, there exists~$h^{(r)} \in \mathbb{H}_\lambda^0$ such that, for every~$s \in [r,\frac12)$, 
\begin{align}  \label{e.ls.fb.c1beta}
 \| u^\ep - h^{(r)}  \|_{L^\infty(B_s)} 
\leq 
C s^2 \Bigr(  s^{\beta}  \inf_{h \in \mathbb{H}_\lambda^0}\| u^\ep -h \|_{L^2(B_{1})} + \bigl( \mathcal{E}(\tfrac{\ep }{r}) \bigr)^\alpha   \Bigl) 
\,  .
\end{align}
\end{theorem}

Motivated by the result of Theorem~\ref{t.fb.regularity}, we may give a definition of a regular free boundary points. Let $\alpha,\sigma,r_0,\gamma$ be from Theorem~\eqref{t.fb.regularity}, 
\begin{definition}[Large-scale regular free boundary point]
\label{d.regularpoint}
Let $\lambda>0$, and ~$u^\ep$ solve 
\begin{align}  \label{e.thm.sol.againone}
\nabla \cdot (\a^\ep \nabla u^\ep)= \lambda \chi_{\{ u^\ep>0\} }  \quad \mbox{in } B_2 \,.
\end{align}
We say that~$x_0 \in \varGamma( u^\ep) \cap B_1$ is a \emph{large-scale regular free boundary point} if 
 \begin{equation} \label{e.regularpoint}
\inf_{s \in [\gamma r,\gamma^{-1} r]}   \frac{| \varLambda(u^\ep) \cap B_s(x_0)|}{|B_s(x_0)|} \geq \vartheta >0 
\quad \mbox{and} \quad 
r \in \big[  \gamma^{-1} R_{\sigma,\alpha}(x_0) , r_0\big] \,.
 \end{equation}
 \end{definition}

\smallskip
 
A few comments are appropriate. First, the heart of the matter is to observe that all constants are independent of $\ep$. This justifies the use of the word ``quantified" in the title. Second, there is no regularity assumed from the matrix $\a$ apart from the uniform ellipticity. The homogenization guarantees that $\mathcal{E}(\tfrac{\ep }{r}) $ can be taken to be small and that it is moreover \emph{quantified}.
In particular, in the periodic setting, $\mathcal{E}(t) = Ct$, so that by taking $r = C \delta^{-1} \ep$, we see that $  \| u^\ep - h  \|_{L^\infty(B_r)}  \leq \delta r^2$ for every $r \in [C \delta^{-1} \ep , c \delta^{1/\beta}]$. One could interpret this as ``large-scale uniqueness of blow-ups" in the sense that the half space solution is fixed for all scales larger than $C\ep$, which implies that also the normal direction of the free boundary is fixed in large-scale regularity sense. Similar types of statements remain valid in the stochastic case, but now the minimal scale~$R_{\sigma,\alpha}$ will be a \emph{random object}.

\smallskip

Finally, there still remains a very interesting open problem, namely, under what condition one would be able to prove (or give a counterexample under some reasonable conditions) that the free boundary would be not only large-scale flat, but flat in the traditional sense. For example, as we were already discussing that in the periodic setting Theorem~\ref{t.fb.regularity} yields large-scale regularity up to scales $C\ep$. Assuming some smoothness of coefficients, Caffarelli's results would imply regularity from $c\ep$ scale on. Thus, the missing range is $[c\ep,C\ep]$. In fact, we suggest the following open problem for future investigation. Assume that the matrix $ \a$ is periodic and sufficiently regular. Can we then obtain uniform regularity results for the free boundary $ \varGamma(u^\ep)$, in a neighborhood of a regular free boundary point?
In this setting, there exist blow-up solutions with respect to $ u^\ep $: for a fixed~$ x_0 \in \varGamma(u^\ep)$ and a subsequence $ r_j \to 0+$ there exists~$ u^\ep_0(x):=\lim_{r_j \to 0+} \frac{u^\ep(r_jx+x_0)}{r_j^2}$, and $ \nabla \cdot \a^\ep(x_0) \nabla u^\ep_0=\lambda \chi_{\{u_0^\ep>0\}}$ in $ \R^d$.
  However, the blow-ups of $ u^\ep $ depend on $ \ep$, and a very careful analysis with respect to both heterogenous and homogeneous  halfspace blow-up solutions will be needed.

\subsection{Organization of the manuscript}

In Section 2 we gather the known homogenization estimates for the Dirichlet problem, and the known regularity theory for the obstacle problem.  An expert in the field  may skip this section.

\smallskip

In Section 3  we obtain a quantitative homogenization result for the solution of the obstacle problem, employing the well-known penalization method for partial differential equations with discontinuous right hand side.  The idea is the following: First we show the closeness of the solution to the obstacle problem to the penalized/regularized problem in the~$H^{1}$-norm. Then we prove that a quantitative homogenization result holds for the regularized problem by following the two-scale expansion argument, and employing the known estimates for first order correctors. In Section 4 we prove large-scale $ C^{1,1}$- regularity for the solution of the obstacle problem with measurable coefficients.
The Section 5 is devoted to the large-scale regularity of the free boundary. 
In the appendix we give two one-dimensional examples, with numerical illustrations, which further justify our assumptions.


 \section{Background theory}
 In this section we gather the known theory of homogenization and of the obstacle problem. It is done for the readers convenience, and an expert in either field can skip this section.

 \subsection{Notations.} 
 
 Throughout the manuscript  $ \R^d$ is
 the $ d$-dimensional Euclidean space and the scalar product is denoted by $ x\cdot y$, for $ x, y \in \R^d$, $ | \cdot |$ is the Euclidean norm. The canonical basis is denoted by $ \{ e_1,...,e_d \}$. The ball with a centre at $ x_0$ and radius $ r >0$ is denoted by $B_r(x_0)$, when $ x_0=0$, we simply write $B_r$, and $ B:=B_1$. We denote, for $m \in \Z$, by $\cu_m$ the triadic cube $\cu_m = (-\tfrac12 3^m , \tfrac12 3^m )^d$.
 
 Let $ U \subset \R^d $, for a scalar valued function $ f : U \to  \R $, we denote the partial derivatives by $ \partial_{x_i} f$. For a vector field $ F:  U \to \R ^d$, we denote the divergence operator
 $$
 \nabla \cdot F:= \sum_{i=1}^d \partial_{x_i} f_i,
 $$
where $ f_i = F \cdot e_i$. The Laplacian of a function $ f$ is denoted by 
$$
\Delta f= \nabla \cdot (\nabla f )=\sum_{i=1}^d \partial_{x_i} ( \partial_{x_i } f).
$$
By $ \nabla^k $ we denote the tensor of the $ k$-th order partial derivatives of $f$.

\subsubsection{Function spaces}
For a natural number $ k \in \N$, and $  U \subset \R^d$ the set of $ k$-times continuously differentiable functions is denoted by   $ C^k$. 
For $ \alpha \in (0,1]$ the H\"older seminorm is the following 
$$ 
[ f ]_{ C^{0,\alpha}(U)}:= \sup_{x, y \in U, x \neq y} \frac{ | f(x)-f(y) |}{ | x-y|^\alpha},
$$
and $ C^{k, \alpha(U) }$ is the H\"older space, with the norm
$$
\| f \|_{ C^{k, \alpha} (U)} :=\sum_{l =0}^k \sup_{ x\in U} | \nabla^l f (x)| + [ \nabla^k f ]_{ C^{0,\alpha}(U)}.
$$

For $ p \in [1, +\infty]$, $L^p(U) $ is the Lebesgue space with the norm
\begin{equation*}
\| f \|_{ L^p(U) } :=\left(  \int _U |f(x)|^p  dx  \right)^\frac{1}{p}, \textrm{ for } p< +\infty, \textrm{ and }
\| f \|_{ L^\infty(U) } := \esssup_{x \in U} |f(x)|.
\end{equation*}

If $ | U| <+\infty$ ($ U $ has a finite Lebesgue measure) then 
$$
(f)_U: =\fint_U   f(x)dx :=\frac{1}{| U|} \int_U f(x) dx,
$$
the average of $f $ on $U$.
We will often use the averaged norms:
\begin{equation*}
\| f \|_{ \underline{L}^p(U) } :=\left(  \fint _U |f(x)|^p  dx  \right)^\frac{1}{p}= 
|U|^{-\frac{1}{p}} \| f \|_{ L^p(U) }, \textrm{ and } \| f \|_{ \underline{L}^\infty(U) } :=\| f \|_{ L^\infty(U) }.
\end{equation*}

The standard notation  $ W^{k,p} (U )$ for Sobolev spaces is used, with the norm
\begin{equation*}
\| f \|_{ {W}^{k,p}(U) }:= \sum_{l=0}^k   \| \nabla^l  f \|_{ {L}^p(U) },
\end{equation*}
and the averaged norm is defined as follows
\begin{equation*}
\| f \|_{ \underline{W}^{k,p}(U) }:= \sum_{l=0}^k  |U |^{\frac{l-k}{d} } \| \nabla^l  f \|_{ \underline{L}^p(U) }.
\end{equation*}
The notation $ W_0^{k,p} (U )$ is used for Sobolev spaces with zero trace. We also use the notation $H^k = W^{k,2}$ for every $k \in \N$.

\subsection{Useful inequalities}

We will also make  use of the following standard interpolation result many times. For the sketch of the proof, see~\cite[Exercise 3.20]{AKMBook}
\begin{lemma} \label{l.interpolation}
Let $\alpha \in (0,1]$ and $r \in (0,\infty)$. Then there exists a constant $C(\alpha,d)<\infty$ such that, for every $u\in C^{0,\alpha}(B_r)$, we have
\begin{align*} 
\| u \|_{L^{\infty}(B_r)} \leq  C \| u \|_{\underline{L}^{2}(B_r)}^{\frac{2\alpha}{d+2\alpha}} \big( r^\alpha [ u ]_{C^{0,\alpha}(B_r)} \big)^{\frac{d}{d+2\alpha}} .
\end{align*}
\end{lemma}
 
 \smallskip

The following classical result can be found, for example in~\cite{HKM}. 
 \begin{lemma}[Meyers lemma] 
 \label{l.Meyers}
 Fix $ r >0$ and $ p \in (2, \infty)$. Suppose that $ f \in W^{-1,p}$ and that $ u \in H^1(B_r)$ satisfy
 \begin{equation}
-\nabla \cdot ( \a \nabla u) =f \textrm{ in } B_r. 
 \end{equation}
 Then there exists $ \delta(p,d, \Lambda) > 0$ and $ C (p,d, \Lambda) <\infty $ such that 
 \begin{align*}
 \|  \nabla u \|_{ \underline{L}^{ (2+\delta) \wedge p} (B_{r/2})}
 \leq C \left(  \|  \nabla u \|_{ \underline{L}^{ 2} (B_{r})} +
 \|  f \|_{ \underline{W}^{ -1, (2+\delta) \wedge p} (B_{r})}
   \right)
 \end{align*}
 Furthermore, the following global version holds: if $u$ solves 
  \begin{equation}
-\nabla \cdot ( \a \nabla u) =f \textrm{ in } U, 
 \end{equation}
 where $U $ is a Lipschitz domain, and $ u \in g +H^{1}_0(U)$, with $ g \in W^{1,p}(U)$,then 
  \begin{align*}
 \|  \nabla u \|_{ \underline{L}^{ (2+\delta) \wedge p} (U)}
 \leq C \left(  \|  \nabla g  \|_{ \underline{L}^{ 2} (U)} +
 \|  f \|_{ \underline{W}^{ -1, (2+\delta) \wedge p} (U)}
   \right)
 \end{align*}
 \end{lemma}

\smallskip

 We shall often refer to the following De Giorgi-Nash-Moser theorem.
 \begin{theorem} \label{t.DGNM} (De Giorgi-Nash-Moser, Theorem 8.22 in \cite{GTbook})
 Let~$p >\tfrac{d}{2}$. 
 Let~$U$ be an open set,~$f\in L^p(U)$ and let~$ u \in H^{1}(U)$ solve 
 \begin{equation} \notag
 \nabla \cdot (\a \nabla u)= f\quad \textrm{in } U,
 \end{equation}
 where $\a$ satisfies the uniform ellipticity assumption~\eqref{e.ellipticity}. Then $ u $ is locally a H\"older continuous function in $U$. In particular, there exist constants $C(p, r_0,d, \Lambda) \in [1,\infty)$ and $\alpha=\alpha(p,d,\Lambda) \in (0,1)$ such that if $B_{2r}(x) \subseteq U$ for some~$r>0$ and~$x\in U$, then, for every $t \in (0,r)$, 
 \begin{equation}  \label{e.DGNM}
 \| u \|_{L^\infty(B_r(x)) }
+
 \Bigl( \frac{r}{t}  \Bigr)^\alpha  \| u - u(x) \|_{L^\infty(B_t(x)) } \leq C \left(  \| u \|_{\underline{L}^2(B_{2r}(x)) } + r^2 \| f \|_{\underline{L}^p(B_{2r}(x))}\right).
 \end{equation}
 \end{theorem}

\smallskip

\subsection{Homogenization theory, assumptions and known estimates}
\label{s.homogenization}
Before proceeding to the discussion of our problem and main results obtained in the manuscript, let us 
give some basic notations, and a short discussion of  key ideas and concepts in the homogenization theory. A very thorough discussion on many aspects of homogenization as well as  rather comprehensive bibliographical references can be found in the monograph of Jikov, Kozlov \& Ole{\u\i}nik~\cite{JKO}.

\subsubsection{Periodic homogenization}

Possibly the most studied branch of homogenization, so far, is the case of deterministic, periodic coefficients, that is, $\a(x + z) = \a(x)$ for every $z \in \Z^d$ and for almost every $x \in \R^d$. Higher dimensional works start from the  works of De Giorgi \& Spagnolo~\cite{DGS} using $\Gamma$-convergence, followed by direct methods of Tartar, see~\cite{Tartar} and the references therein, and two-scale analysis of Bensoussan \& Lions \& Papanicolaou~\cite{BLP}. For the contemporary analysis we refer to an excellent monograph of Shen~\cite{ShenBook}. In the periodic case, we find~$\Z^d$-periodic correctors~$\phi_{e_k}$ and~$\mathbf{S}^{(e_k)}$ solving~\eqref{e.corrector.intro} and~\eqref{e.fluxcorrector.intro} with $m = 0$, respectively, so that~$\phi_{e_k,m}$ and~$\mathbf{S}_m^{(e_k)}$ are simply periodic extensions of these. In particular, there exists a constant~$C(d,\Lambda) <\infty$ such that, for every~$\ep \in (0,1]$, 
\begin{align}  \label{e.error.periodic}
\mathcal{E}(\ep) = C \ep .
\end{align}
Homogenized matrix can be defined using the flux of the corrector, that is, for every~$e \in \R^d$, 
\begin{align}  \label{e.ahom.periodic}
\ahom e = \int_{\cu_0} \a(x) (e + \nabla \phi_e(x)) \, dx ,
\end{align}
and  we may set $\ahom_m = \ahom$.

\subsubsection{Almost periodic homogenization}

Homogenization of elliptic equations with almost-periodic coefficients was studied first by Kozlov~\cite{K0}. The quantified theory was initialized in~\cite{Sh1} and later on with large-scale regularity and more quantified rates in~\cite{ASh}. The necessary condition for algebraic rate~\eqref{e.algebraic}  is that there exists a constant~$\beta \in (0,1)$ such that, for $R \gg 1$, 
\begin{align*} 
 \sup_{y \in \R^d} \inf_{z \in B_R } \| \a(\cdot+y) - \a(\cdot+z)\|_{L^\infty(\R^d)} 
 \leq  
 R^{-\beta}
\end{align*}
Under stronger quantitative assumptions on $\a$ it is also possible to prove~\eqref{e.error.periodic}, see~\cite{AGK}.

\subsubsection{Stochastic homogenization}

Stochastic homogenization has been under intensive research for the last 10-15 years. Early contributions in qualitative theory were made already earlier, in 1980s, by Yuriinski~\cite{Y0}, Kozlov~\cite{K1} and Papanicolaou \& Varadhan~\cite{PV1}. Let us also mention here the variational methods developed by Dal Maso \& Modica~\cite{DM1,DM2}. The starting point for the quantitative theory, which is also the backbone of this paper, was the unpublished manuscript of Naddaf \& Spencer. The main tools in this paper were  concentration inequalities, such as the Efron-Stein or logarithmic Sobolev inequalities. These ideas were pushed forward by Gloria \& Otto, starting from~\cite{GO1}. They, together with Fischer \& Neukamm and many others~\cite{GNO,FO}, developed  the theory of the quantitative homogenization, see references  in~\cite{GNO2}. Later, a different approach was taken by Armstrong \& Smart~\cite{AS},  pushing much further the variational approach of Dal Maso \& Modica. The important contribution of this paper was that it developed the large-scale regularity following Avellaneda \& Lin in the context of stochastic homogenization. Moreover, these techniques were developed further by Armstrong \& Mourrat~\cite{AKM1,AKM2} and the second author of this paper, and much of the theory is summarized in the monograph~\cite{AKMBook}. See also~\cite{GO6}. 

\smallskip

Let~$U \subset \Rd$ be a Lipschitz domain, and~$p\subset \Rd$, and~$\ell_p(x):=p\cdot x$.
Denote by 
\begin{align}
\mu(U,p):= \inf_{v \in \ell_p+H_0^1(U)} \fint _U \frac{1}{2} \nabla v \cdot \a \nabla v, 
\end{align}
and by~$v(\cdot, U,p)$ the unique minimizer of this energy. Then 
\begin{align}
\notag
-\nabla \cdot ( \a \nabla v(\cdot,U,p))=0 \textrm{ in } U, \quad v = \ell_p \textrm{ on } \partial U.
\end{align}
It is easy to show that the mapping~$p\mapsto \mu(U,p)$ is a positive quadratic form and subadditive with respect to partitions of~$U$. We denote by $\a_m$ the matrix having components 
\begin{equation} \label{e.ahom}
(\a_m)_{ij} 
= 
\fint _{\cu_m} \frac{1}{2} \nabla v(\cdot,\cu_m,e_i) \cdot \a \nabla v(\cdot,\cu_m,e_j)
.
\end{equation}
Under some mild ergodicity assumption, the limit $\ahom = \lim_{m \to \infty} \a_m$ exists almost surely. The matrix~$\ahom~$satisfies the same ellipticity conditions as~$\a$, and it is called as the homogenized matrix.

\smallskip

For every~$m \in \N$ and~$e \in \R^d$, we set
\begin{equation} \label{e.corrector.m}
\phi_{m,e}(x):= v(x,\cu_m,e)-e\cdot x.
\end{equation}
Observe that~$\phi_{m,e}$ satisfies the equation~\eqref{e.corrector.intro} in~$\cu_m$. Having defined the finite volume corrector~$\phi_{m,e}$ and extending it periodically to the whole~$\R^d$, we define the finite volume flux corrector as a~$3^m \Z^d$-periodic function~$\mathbf{S}_m^{e}$ such that its components are, for every~$i,j \in \{1,\ldots,d\}$, unique solutions of the equations
\begin{align} \label{e.fluxcorrector}
\Delta \mathbf{S}_{m,ij}^{e} = \partial_{x_j} \big( \a ( e+ \nabla \phi_{m,e})\big)_i - \partial_{x_i} \big( \a ( e+ \nabla \phi_{m,e})\big)_j, \quad
\mathbf{S}_{m,ij}^{e} \in H_{\mathrm{per}}^1(\cu_m)
\end{align}
Indeed, we see by taking divergence that~$\sum_{j=1}^d \partial_{x_j} \mathbf{S}_{m,ij}^{e} - \big( \a ( e+ \nabla \phi_{m,e})\big)_i$ is a periodic harmonic function in~$\R^d$, and as such it is a constant. Since the equation is invariant under constant addition, we may take that constant to be~$-(\ahom e)_i$. Moreover, we have the following estimate for~$\mathbf{S}_{m,ij}^{e}$, for every~$i,j \in \{1,\ldots,d\}$ and~$e \in B_1$, 
\begin{align} \label{e.fluxcorrector.est1}
\left\| \mathbf{S}_{m,ij}^{e} \right\|_{\underline{L}^2(\cu_m)} \leq C 
\left\| \a ( e+ \nabla \phi_{m,e}) - \ahom_m e \right\|_{\hat{\underline{H}}^{-1}(\cu_m)} 
,
\end{align}
where the weak norm on the right is defined as
\begin{align*} 
\|  f \|_{\hat{\underline{H}}^{-1}(\cu_m)} 
:=
\sup \left\{\fint_{\cu_m} f g \, : \, g \in H^1(\cu_m), \; 3^m \| g \|_{\underline{L}^2(\cu_m) } + \| \nabla g \|_{\underline{L}^2(\cu_m) } \leq 1  \right\}
.
\end{align*}

\smallskip

We next give an example of assumptions in stochastic homogenization, namely the case of finite range dependence. 
Endow~$\Omega$, the set of symmetric matrices with eigenvalues on the interval~$[1,\Lambda]$, with a family of~$\sigma$-algebras~$\left\{ \F_U \right\}$ indexed by the family of Borel subsets~$U\subseteq \Rd$, defined by
\begin{equation}
 \label{e.FU.def}
\begin{aligned}
\F_U & :=  \mbox{the~$\sigma$-algebra generated by the following family:}  \\
& \quad \quad \left\{ \a \mapsto \int_{U} \a_{ij} (x) \varphi(x)\,dx \, : \, \varphi \in C^\infty_c(U), \ i,j\in \{1,\ldots,d\} \right\}.
\end{aligned}
\end{equation}
The largest of these~$\sigma$-algebras is denoted~$\F:= \F_{\Rd}$. For each~$y\in\Rd$, we let~$T_y:\Omega \to\Omega$ be the action of translation by~$y$,
\begin{equation} 
\label{e.Ty}
\left( T_y\a\right)(x):= \a(x+y),
\end{equation}
and extend this to elements of~$\F$ by defining~$T_y E:= \left\{ T_y\a\,:\, \a\in E \right\}$. Assume that the probability measure~$\P$ on the measurable space~$(\Omega,\F)$ satisfy the following two assumptions. 

\smallskip

\begin{itemize} 

\item \emph{Stationarity with respect to~$\Zd$-translations:}
\begin{equation} 
\label{e.ass.stationarity}
\P \circ T_z = \P \quad \mbox{for every} \ z\in\Zd. 
\end{equation}

\item \emph{Unit range of dependence}. 
\begin{equation}
\label{e.ass.independence}
\begin{aligned}
& \F_U \ \mbox{and} \ \F_V \ \  \mbox{are~$\P$-independent for every pair~$U,V\subseteq \Rd$}  
\\
& \mbox{of Borel subsets satisfying~$\dist(U,V)\geq 1$.}
\end{aligned}
\end{equation}

\end{itemize}

We denote the expectation with respect to~$\P$ by~$\E$. That is, if~$\X:\Omega \to \R$ is an~$\F$-measurable random variable, we write
\begin{equation} 
\label{e.E.def}
\E \left[ \X \right] := \int_\Omega \X(\a) \, d\P(\a). 
\end{equation}
While most of the objects in our study are functions of~$\a\in\Omega$, we do not typically display this dependence explicitly in our notations. 
For a random variable~$\X$ and parameters~$s, \theta \in  (0, \infty)$, we use the abbreviation 
\begin{equation}
\X \leq \O_s(\theta) 
\qquad 
\Longleftrightarrow
\qquad 
\E \left[ \exp(\theta^{-1} \X_+)^s \right] \leq 2 .
\end{equation}
Now, following~\cite{AKMBook}, one can show that under these assumptions, for $s \in (0,d)$, there is a constant $C(s,d,\Lambda)$, a random variable $\X$ satisfying $\X \leq \O_s(C)$ and an exponent $\gamma(d,\Lambda) \in (0,1]$ such that
\begin{align*} 
\mathcal{E}(\ep) \leq ( \X \ep )^\gamma    .
\end{align*}
There are of course other mixing conditions allowing the above estimate for the sublinearity of the correctors, see for example~\cite{GO6}.

\smallskip

We will conclude the discussion about homogenization with an important remark. In the random setting, $\mathcal{E}(\ep)$ itself is a random object. Under stationarity assumptions, it has its stationary extension. Notice that our results are all centered at the origin. When considering the translations of these results, which are of course trivial in the case of Laplacian or periodic coefficients, one has to be careful in the random setting; the minimal scale and $\mathcal{E}(\ep)$ will change, in general, from point to point, see~\eqref{e.error.translation}.


\subsection{The obstacle problem}
\label{s.classics}
First we give a brief summary of the classical obstacle problem, following  L. A. Caffarelli (\cite{Caf77,Caf98}).
The main purpose here is to describe the main tools and techniques in investigation of the regularity of the free boundary.

\smallskip
As before~$U \subset \mathbb{R}^d$ is  a given Lipschitz domain  and~$\varphi \in H^{1}(U)$ is the obstacle.  Then the minimizer to the following functional 
\begin{equation} \label{J}
 J[u]= \int_\Omega | \nabla u(x) |^2 dx,
\end{equation}
over all functions~$u=g~$,~$\varphi < g$ on~$\partial U$,   such that~$u\geq \varphi$ in~$U$,  is called the solution to the 
obstacle problem with obstacle~$\varphi$. The solution
satisfies the following variational
inequalities
\begin{equation} \label{varineq}
\Delta u\leq 0, ~ u \geq \varphi, ~ \Delta u \cdot (u-\varphi)=0.
\end{equation}
The following set~$\varOmega(u):= U \cap \{ u>\varphi\}$ is called the non-coincidence set, and 
$\varGamma(u):=\partial \Omega(u) \cap U$ is called the free boundary. 

It is well known that the solution to the obstacle problem is as regular as the obstacle up to~$C^{1,1}$, \cite{Frehse}.The solution~$u$ cannot obtain a better regularity, since~$\Delta u$ is discontinuous on~$\varGamma(u)$. 
 Let~$v:= u-\varphi  \geq 0$, then $\Delta v =- \Delta \varphi \chi_{ \{ v >0 \}}$. Denoting by $ f:=- \Delta \varphi \geq 0$, we derive the 
 normalized obstacle problem
\begin{equation}
\Delta v =f  \chi_{ \{ v >0 \}}.
\end{equation}

It is clear that the properties of the normalized obstacle problem above can be easily adapted  for the following obstacle problem,  $\nabla (\ahom \nabla \bar{u} )=f \chi_{\{\bar{ u} >0\}} $, with the homogenized matrix $ \ahom$.

\begin{theorem} \label{t.Frehse} [Frehse,\cite{Frehse}]
Let $ \bar{u} \geq 0$ be the solution to the obstacle problem $  \nabla (\ahom \nabla \bar{u} )=f \chi_{\{\bar{ u} >0\}}$ in a domain $ U$.
Assume that $B_{2r}(z) \subset U$ and there exists a
$ \psi \in C^{1,1}$, such that $ \nabla (\ahom \nabla \psi )=f$, then 
\begin{align} \label{e.optreg}
\| D^2 \bar u \|_{L^\infty(B_{r})} \leq C(d, \Lambda) \left( \frac{\| \bar u \|_{L^\infty(B_{2r})}}{r^2} + \| D^2 \psi \|_{L^\infty (B_{2r})}\right), 
\end{align}
In particular,  if $f \in C^{0,\beta}(B_{2r}(z))$, then there exists a constant $C(\beta,d,\Lambda)< \infty$ such that
\begin{align} \label{e.optregH}
\| D^2 \bar u \|_{L^\infty(B_{r})} \leq C(\beta, d, \Lambda) \left( \frac{\| \bar u \|_{L^\infty(B_{2r})}}{r^2} +\| f \|_{L^\infty(B_{2r})} + r^\beta [ f ]_{C^{0, \beta}(B_{2r})}\right).
\end{align}

\end{theorem}

The proof of~\eqref{e.optreg} can be easily found, for instance in \cite{PSU}, and~\eqref{e.optregH} follows by the Schauder estimates.

\smallskip
The regularity of the free boundary~$\varGamma(u)$  has been well studied in the literature \cite{Caf77, Caf98}. It is customary to take~$f \equiv 1$, since the more general case~$f \geq c>0$, under regularity assumptions on~$f$, can be investigated following similar techniques, \cite{B01}.
Fix~$x_0 \in  \varGamma(u)$, then the solution 
has a quadratic decay at~$x_0$
\begin{equation} \label{quadratic}
\sup_{B_r(x_0)} u\leq C r^2.
\end{equation}
On the other hand for any~$x_0 \in \overline{ \{ u>0 \} } \cap U$, and any~$B_r(x_0) \subset U$,
\begin{equation} \label{nondeg}
\sup_{B_r(x_0)} u\geq c r^2,
\end{equation}
which is the non-degeneracy property.

Denote by 
\begin{equation*}
u_{r,x_0}:=\frac{u(x_0 +rx )}{r^2},
\end{equation*}
then~$u_{r_j, x_0} \to u_0$ up to a subsequence~$r_j \to 0+$. The limit function~$u_0$ is a global solution to the obstacle problem~$\Delta u_0=\chi_{\{ u_0>0\}}$, and is called a \emph{blow-up solution}. We can understand the structure of the free boundary at~$x_0$ by analyzing the blow-up limits at that point.

\smallskip
\begin{theorem}\label{t.Caffarelli.dichotomy} [Caffarelli~\cite{Caf77,Caf98}]
 Let $ \bar{u}$  be the solution to the obstacle problem $  \nabla (\ahom \nabla \bar{u} )= \chi_{\{\bar{ u} >0\}}$, and $ x_0$ be a free boundary point. Then the blow-up of~$\bar{u}$ at~$x_0$, $u_0$, is unique, and there are two possibilities. Either~$u_0$ is a half-space solution, that is, 
 \begin{equation} \label{regularpoint}
  u_0(x)=\frac{(x\cdot e )_+^2}{2 e \cdot \ahom e}, \textrm{ where } e \in \mathbb{R}^d \textrm{ is a unit vector},
 \end{equation}
 and~$x_0$ is a called a regular free boundary point. In a neighborhood of each regular free boundary point~$\varGamma(u)$ is an analytic surface.
 Otherwise,~$u_0$ is a second order nonnegative polynomial,
 \begin{equation}
  u_0(x)=\frac{(x\cdot A x )}{2 \tr(\ahom A)}, \mbox{ where 
  $A \in \R^{d \times d}$ is a nonnegative symmetric matrix.} 
 \end{equation}
In this case,~$x_0$ is called a singular free boundary point, and all such points are lying on a lower dimensional~$C^1$-manifold.
  \end{theorem}

 \smallskip
 
\subsubsection{The divergence form obstacle problem  with measurable coefficients}
Now we consider the following obstacle problem 
\begin{equation}
\label{e.hetereqagain}
\nabla \cdot (\b \nabla u)= f \chi_{ \{ u>0\}}, \qquad  f \geq \lambda > 0,
\end{equation}
where~$\b$ is an elliptic ($I_d \leq \b \leq \Lambda I_d$) and symmetric matrix, with measurable coefficients, and~$f$ is a measurable function satisfying the lower bound in~\eqref{e.hetereqagain}.

\begin{lemma} \label{l.basicregularity}
Fix  $q \in (\frac d2,\infty)$ and $\delta \in (0,1)$. 
Suppose that~$g \in W^{1,2+\delta}(U)$ and~$f \in L^q(U)$ for $q \in (\frac d2,\infty)$

\begin{itemize}
\item Global estimates. There exists constants $C(U,d,\Lambda)< \infty$ and $\alpha(\delta,d,\Lambda)\in (0,1)$ such that 
\begin{align} \label{e.Meyers}
\left\| \nabla u^\ep  \right\|_{L^{2+\alpha} \left(U \right)} 
+
\left\| \nabla u  \right\|_{L^{2+\alpha} \left(U \right)} 
\leq 
C \big( \| \nabla g \|_{L^{2+\delta}(U)}+\|f\|_{L^{2+\delta}(U)} \big).
\end{align}
\item Local estimates. There exists a constant $C(d,\Lambda)< \infty$ and $\alpha(d,\Lambda) \in (0,1)$ such that if $B_{2r}(z) \subset U$, then
\begin{multline}  \label{e.holder}
\left\| u^\ep  \right\|_{L^{\infty} \left(B_{r}(z) \right)} 
+
r^\alpha \left[ u^\ep  \right]_{C^{0,\alpha} \left(B_{r}(z) \right)} 
+ r \left\| \nabla u^\ep  \right\|_{L^{2+\alpha} \left(B_{r}(z) \right)} 
\\ 
\leq 
C \big( \left\| u^\ep - (u^\ep)_{B_{2r}(z)}  \right\|_{\underline{L}^{2} \left(B_{2r}(z) \right)} + r^2 \|f\|_{\underline{L}^{q}(B_{2r}(z))} \big)
.
\end{multline}
\end{itemize}

\end{lemma}

Here~\eqref{e.Meyers} follows by treating $f \chi_{\{u^\ep > 0\}}$ as  bounded right hand side. The H\"older regularity estimate in the scaling invariant form is also classical, see~\cite{HKM}.

\smallskip

Next we would like to state the properties of the obstacle with rough coefficients, following  Ivan Blank and collaborators, \cite{ BZ1, BZ2, BK}.

\begin{lemma}[Blank \& Hao~\cite{BZ1}]
\label{l.BlankHao}
Let $ I_d \leq  \a \leq \Lambda I_d $, 
be a symmetric matrix, with bounded measurable coefficients and  $ f \geq \lambda >0  $
be a bounded measurable function. There exist constants~ $0 < c <C< \infty $, both depending only on the given data  $(\lambda,d,\Lambda)$, such that for any  $ u \geq 0 $ solution of the divergence-form obstacle problem 
\begin{equation}
\nabla \cdot(\a \nabla u)=f \chi_{\{ u>0\}}, \textrm{ in } U
\end{equation}
the following statements hold:\\
 \textit{Quadratic decay (\cite[Theorem 3.2]{BZ1}):} If $ x_0 \in \varGamma(u)$ is  a fixed free boundary point, then 
 \begin{equation} \label{e.quadratic}
\sup_{B_r(x_0)} u \leq C\| f  \|_{L^\infty(U)} r^2.
\end{equation}
  \textit{Non-degeneracy (\cite[Theorem 3.9]{BZ1}):}
If $ x_0\in \overline{\{ u >0 \} } \cap U $, then 
    \begin{equation} \label{e.nondeg}
 \sup_{B_r(x_0)}  u \geq c r^2,
\end{equation}
provided $ B_{2r}(x_0)\subset U$.
\end{lemma}

If $ \a =I_d$, then  in \cite{B01} Blank gives an example, showing that the blow-ups are not unique if $ f $ is merely continuous. If $ a_{ij}$ and $ f $ are merely bounded and measurable, we cannot even define the blow-up solutions, see Section 5. 
There is a counterexample in~\cite{BZ1} showing  that if~$\a$ is discontinuous then~$\lim_{r\to 0+} u_r$ does not exists, even when imposing a  uniform positive density assumption on the coincidence set $ \{u \equiv 0\}$ at free boundary points.  In \cite{BZ2}, see also \cite{BK}, the  authors  show the existence of blow-ups assuming that $ \a_{ij}, f \in VMO$, they also show that in general the blow-ups are not unique by constructing a counterexample.


\smallskip

\section{Homogenization of the obstacle problem }

In this section we will prove the necessary homogenization result for the obstacle problem. 
Let  $U$ be a Lipschitz domain,~$g \in W^{1,2+\delta}(U)$ and~$f \in L^\infty(U)$,~$f > 0$~a.e.. 
We consider the nonnegative solutions $ u^\ep$ and $ u$ to the following obstacle problems respectively 
\begin{align}  \label{e.homog.eqs}
\left\{
\begin{aligned}
& \nabla\cdot \a^\ep \nabla u^\ep = f \chi_{\{ u^\ep > 0 \}} & & 
\mbox{ in } U,
\\
& u = g & & 
\mbox{ on } \partial U,
\end{aligned}
\right.
\qquad 
\left\{
\begin{aligned}
& \nabla\cdot \ahom \nabla u = f \chi_{\{ u > 0 \}} & & 
\mbox{ in } U,
\\
& u = g & & 
\mbox{ on } \partial U
.
\end{aligned}
\right.
\end{align}
We will derive the homogenization result for the obstacle problem in two steps. First, we compare the solution of the obstacle problem with the solution of the penalized problem, for which merely measurable coefficients suffice. Second, we prove the actual homogenization estimates for the penalized problem, and after that tie these together simply by the triangle inequality. 

\smallskip 

The main theorem of this section reads as follows. 
\begin{theorem} \label{t.homogenization}
Let~$U$ be a bounded Lipschitz domain in~$\cu_0$. Let~$\delta \in (0,\infty]$. There exist constants~$C(U,\delta,d,\Lambda)<\infty$ and~$\alpha(\delta,d,\Lambda)  \in (0,\tfrac12)$ such that the following statement is valid. Suppose that~$g \in W^{1,2+\delta}(U)$ and~$f \in L^\infty(U)$ are such that~$g \geq 0$ and~$f>0$ almost everywhere in~$U$. 
Let~$u^\ep \geq 0$ and~$u \geq 0$ solve the obstacle problems~\eqref{e.homog.eqs} respectively.
Then we have the estimate
\begin{multline} \label{e.homogenization}
\| u^{\varepsilon} - u \|_{L^2(U)} 
+\| \nabla u^{\varepsilon} - \nabla u \|_{H^{-1}(U)} 
+\| \a(\tfrac{\cdot}{\ep})\nabla u^{\varepsilon} - \ahom \nabla u \|_{H^{-1}(U)} 
\\
\leq C \big( \mathcal{E} (\ep)\big)^\alpha ( \| \nabla g \|_{L^{2+\delta}(U)}+\|f\|_{L^\infty(U)}).
\end{multline}
\end{theorem}

\subsection{Closeness of the two obstacle problems}
We claim  that the homogenization  of the solution to the normalized obstacle problem imply homogenization of the solution to the original obstacle problem, i.e. with given obstacle $\varphi \in H^1(U)$. 
For this, consider the problems 
\begin{align} \label{e.varhpieqs}
\left\{
\begin{aligned}
 & \nabla \cdot \a^\ep \nabla w^\ep =  \nabla \cdot \ahom \nabla w
&& 
  \mbox{in } U , 
\\
& w^\ep =  w
&& 
  \mbox{on } \partial U 
\end{aligned}
\right.
\end{align}
with $w\in W^{1,2+\beta} (U)  \cap W_{\mathrm{loc}}^{2,\infty} (U)$, $\beta \in (0,\infty]$.

\begin{lemma} 
\label{l.wepandw.close}
Fix $\beta\in (0,1)$ and let $U \subseteq B_1$ be a Lipschitz domain. Then there exist constants $\delta(\beta,d,\Lambda) \in (0,\beta]$ and $C(U,\beta,d,\Lambda)< \infty$ such that the following holds.  If~$w^\ep \in W^{1,2+\beta} (U)$ 
and $w \in W^{1,2+\beta} (U)  \cap W_{\mathrm{loc}}^{2,\infty} (U)$ solve~\eqref{e.varhpieqs}, then, for every~$r \in (0,1)$, 
\begin{align} \label{e.varhpisclose}
\| w^\ep - w\|_{L^{2} (U)} 
& 
\leq 
C \big(\mathcal{E}(\ep)  +  r^{\frac{\delta}{2+\delta} } \big) \| \nabla w \|_{L^{2+\delta} (U)}
\\ \notag
& \quad
+ 
C  \mathcal{E}(\ep) 
\big(
r^{-1} \| \nabla w \|_{L^{\infty} (U_r)}   +  \| \nabla^2 w \|_{L^{\infty} (U_r)}  
\big) 
,
\end{align}
where $U_r  = \{ x\in U \, : \, \dist(x,\partial U) > r \}$. 
\end{lemma}
\begin{proof}
First, define the two-scale expansion of $\varphi$ as
\begin{align*} 
\theta^\ep(x) = w(x) + \ep \eta \sum_{k=1}^d \partial_{x_k} w(x) \phi_{e_k}\big(\tfrac{x}{\ep} \big)
\end{align*}
with $\eta \in C_0^\infty(U)$ being a cutoff function such that $\indc_{U_r} \leq \eta \leq \indc_{U}$.  By a direct computation (see~\cite[Chapters 1 \& 6]{AKMBook} or the computation in the proof of Proposition~\ref{p.homogenization})
\begin{align*} 
\nabla \cdot \a^\ep \nabla \theta^\ep = \nabla \cdot \ahom_r \nabla w + \nabla \cdot \g_\ep
\,,
\end{align*}
where 
\begin{align*} 
\g_i^\ep =  \sum_{j=1}^d  (1-\eta)(\a_{ij}^\ep - \ahom_{r,ij})
\partial_{x_j} w +  \sum_{j,k=1}^d \partial_{x_j} (\eta \partial_{x_k} w) (- \mathbf{S}_{m,ij}^{\ep,e_k} + \a_{ij}^\ep \phi_{e_k}^\ep ).
\end{align*}
By equations of $w^\ep$ and $w$, the Poincar\'e inequality, and the fact that $w^\ep = \theta^\ep$ on $\partial U$, we obtain that
\begin{align*} 
\| w^\ep-  \theta^\ep  \|_{H^{1}(U)} 
\leq 
C \| \g^\ep \|_{L^{2}(U)} \,. 
\end{align*}
Therefore, by the triangle inequality and~\eqref{e.error.intro},  we arrive at 
\begin{align*} 
\| w^\ep-  \theta^\ep  \|_{H^{1}(U)}  
\leq 
C \big(\mathcal{E}(\ep)  {+}  r^{\frac{\delta}{2+\delta} } \big)  \| \nabla w \|_{L^{2+\delta} (U)}
+
C  \mathcal{E}(\ep) 
\big(
r^{-1} \| \nabla w \|_{L^{\infty} (U_r)}   {+}  \| \nabla^2 w \|_{L^{\infty} (U_r)}  
\big) 
.
\end{align*}
Applying then once more the bound~\eqref{e.error.intro} for $\phi_{e_k}$ completes the proof.
\end{proof}

We also record the comparison principle for the solutions. 

\begin{lemma}[Comparison principle] \label{l.comparison}
Let $U$ be a Lipschitz domain and $g\in H^1(U)$.  Suppose that $\varphi$ and $\tilde \varphi$ are two continuous obstacles in $H^1(U)$ and let $u$ and $\tilde u$ be  solutions to the obstacle problem with boundary values $g$ on $\partial U$ and  continuous up to the boundary. Then
\begin{align*} 
\| u - \tilde u \|_{L^\infty(U)} \leq \| \varphi - \tilde \varphi \|_{L^\infty(U)} .
\end{align*}
\end{lemma}

\begin{proof}

Let $\delta = \| \varphi - \tilde \varphi \|_{L^\infty(U)}$ and suppose that $\delta > 0$. Assume, on the contrary to the claim, that there is a point $x \in U$ such that $  u(x) > \tilde u(x) + \delta$. Then, by continuity, there is a domain $V$ in the interior of~$U$ such that $u > \tilde u + \delta$ in~$V$ and $u = \tilde u + \delta$ on~$\partial V$. Indeed, $u = \tilde u$ on $\partial U$, and hence~$\partial V$ cannot touch $\partial U$. Now, in $V$ we have that $u$ is above $\varphi$ and hence it is a weak solution in $V$. On the other hand, $\tilde u + \delta$ is a weak supersolution in $V$, and since $\tilde u + \delta = u$ on $\partial V$, the minimum principle implies that $  \tilde u + \delta \geq u$ in $V$; a contradiction. Thus the claim follows. 
\end{proof}

\subsection{Penalized obstacle problem}

The  penalization method is a well-known argument, which is widely used in the regularity theory of  variational problems, see for example \cite{Fribook}. The core idea of the method is the following: Given a solution $u$ to a partial differential  equation with a discontinuous right hand side, we consider the solution of the corresponding partial differential  equation with a regularized right hand side, with the same boundary values as the given function $u$. Showing that  the given solution $ u$ is approximated by a more regular function,  we then  obtain regularity estimates for the given function $u$. 

In this manuscript we use the penalization method in order to obtain quantitative homogenization  for the obstacle problem. Below we describe the method for the obstacle problem, and derive the approximation lemma, even though we believe it can be found in the literature.

\smallskip

Let~$U \subset \mathbb{R}^d$ be a given open bounded domain, with Lipschitz boundary, and~$\b$ be a given uniformly elliptic matrix, with measurable coefficients in~$U$, satisfying the ellipticity condition~$|\xi|^2 \leq \b(\cdot) \xi \cdot \xi \leq \Lambda |\xi|^2$. Let~$f \in L^\infty(U)$ be a given bounded function satisfying~$f >0$ a.e.. Denote by~$u$ the unique minimizer of the following functional
\begin{equation}
J[v, \b]=\int_U \frac{1}{2}\nabla v \cdot (\b \nabla v ) +f v \, dx 
\end{equation}
over~$v \in W^{1,2}(U)$,~$v \geq 0$, satisfying the boundary condition~$v=g>0$ on~$\partial U$.
Then~$u~$satisfies the following variational inequality in a weak sense
\begin{equation} \label{e.equationA}
\nabla \cdot (\b \nabla u )=f\chi_{\{u>0\}}, ~u \geq 0 \textrm{ in } U,~ u=g \geq 0 \textrm{ on } \partial U.
\end{equation}
We see that the right hand side in the partial differential equation in~\eqref{e.equationA} is a discontinuous function. The penalization method is studying an approximating problem for~$~\eqref{e.equationA}$, with regularized right hand side. For this, let~$\beta :(-\infty, +\infty)\to [-1,1]$ be an approximation of the function~$\chi_{\{v>0\}}$, for simplicity, we take ~$\beta$ to be the following Lipschitz continuous function,
\begin{equation} \label{e.beta}
\begin{aligned}
\beta(t)=t~ \textrm{ for }~ |t| \leq 1 
\quad \textrm{ and } \quad 
\beta({t})= t/|t| ~ \textrm{ for } ~ | t|  > 1,
\end{aligned}
\end{equation}
and denote by~$\gamma \in C^{1,1}$ its integral function~$\gamma(t) := \int_0^t \beta(t') \, dt'$. 
Let
\begin{align}
\beta_s(t):=\beta(\tfrac{t}{s}) \quad \textrm{ and } \quad \gamma_s(t)=s\gamma(\tfrac{t}{s})
\quad \textrm{ so that } \quad 
\gamma_s^\prime(t)= \beta_s(t) .
\end{align}

\smallskip

Denote by~$u^s$ the unique minimizer of the following energy
\begin{equation}
J_s[v, \b]=\int_U \tfrac{1}{2}\nabla v \cdot (\b \nabla v ) +f\gamma_s(v)dx 
\end{equation}
over~$v \in H^{1}(U)$,  satisfying the boundary condition~$v=g \geq 0$ on~$\partial U$, or equivalently 
\begin{equation}
 \nabla \cdot (\b \nabla u^s) =f \beta_s( u^s), ~u^s =g \geq 0\textrm{ on } \partial U.
\end{equation}

\smallskip

\begin{lemma} \label{l.penalization}
Let~$U$ be a Lipschitz domain. Let~$q  \in ( \frac d2, \infty]$ and let~$f \in L^q(U)$ be a given function such that~$f$ is positive almost everywhere in~$U$ and let~$g \in H^1(U)$ be a nonnegative function in~$U$. Let~$\b$ be a symmetric matrix satisfying the ellipticity condition~$|\xi|^2 \leq \b(\cdot) \xi \cdot \xi \leq \Lambda |\xi|^2$.
If~$u$ is the solution to the obstacle problem
\begin{align}
\notag
 \nabla \cdot( \b \nabla u )=f\chi_{\{ u>0 \}},
 \notag 
 \quad &
   u - g \in H_0^1(U)
 \end{align}
 and~$u^s$ solves the penalized equation 
\begin{equation}
 \nabla \cdot (\b \nabla u^s) =f \beta_s( u^s), ~ u^s - g \in H_0^1(U),
\end{equation}
then~$u^s$ is a nonnegative continuous functions and 
\begin{equation}\label{e.approx}
 0 \leq  \nabla \cdot (\b \nabla u^s) \leq f \textrm{ in a weak sense}.
\end{equation}
Furthermore,~$u^s \to u$ as~$s\to 0+$ with the following quantitative estimates:
\begin{equation}\label{e.penalized}
0  \leq u^s-u\leq  s  \quad \textrm{in } U
\end{equation}
and 
\begin{equation}\label{e.penalizedgradient}
\| \nabla  u^s-\nabla u\|^2_{L^2(U)}\leq    \| f\|_{L^1(U)} s.
\end{equation}
\end{lemma}

\begin{proof}
\emph{Step 1.}
We first argue that both~$u$ and~$u^s$ are H\"older continuous functions for every~$s>0$. This is implied by the De Giorgi-Nash-Moser theory since both~$f \chi_{\{u>0\}}$ and~$f \beta_s( u^s)$ belong to~$L^q(U)$ with~$q>\frac d2$. Notice also that~$x \mapsto  \beta_s( u^s(x))$ is H\"older continuous as well.

\smallskip

\emph{Step 2.} We next show that~$u^s$ is a nonnegative function, giving also~\eqref{e.approx}. To this end, let~$x_0~$ be a point of minimum 
for~$u^s$ and suppose, on the contrary, that~$u^s(x_0) < 0$. This implies that~$x_0$ is an interior point of~$U$ since~$g \geq 0$. By H\"older regularity of~$u$ and non-negativity of the boundary values, there exists positive~$r=r(s, x_0)$ such that~$u^s<0$ in~$B_r(x_0)$. It thus follows that 
\begin{equation*}
 \nabla \cdot (\b \nabla u^s)  =f \beta_s(u^s) <0 ~\textrm{ a.e. in } ~ B_r(x_0), 
 \end{equation*}
which means that~$u^s$ is a continuous weak supersolution in~$B_r(x_0)$ with an interior minimum point. By the strong minimum principle, implied by the weak Harnack inequality, we hence obtain that~$u^s$ has to be  a negative constant in~$B_r(x_0)$, which violates the fact that~$f \beta_s(u^s)<0$ a.e. through the equation; a contradiction. Therefore~$u^s$ is nonnegative in~$U$ and, consequently,~\eqref{e.approx} is valid. 

\smallskip

\emph{Step 3.}
We then show that~\eqref{e.penalized} holds. Let~$w^s:= u^s-u$. 
First we show that~$w^s \leq s$. For this, let~$E:= \{ u^s> s \} \cap U$ and observe that if~$u^s(x) \leq s$ for some~$x\in U$, then~$w^s(x) \leq s$ as well since~$u \geq 0$. 
Thus~$w_s \leq s$ on~$U \setminus E$. Moreover,~$w^s \leq s$ on~$\partial E$ and 
\begin{equation*}
 \nabla \cdot(\b \nabla  w^s)=f( 1- \chi_{\{u>0\}} )\geq 0~ \textrm{ in }~ E. 
 \end{equation*}
Therefore~$w_s$ is a continuous subsolution in an open set~$E$ and it attains its maximum on the boundary of~$E$, which implies~$w^s \leq s$ in the whole~$E$. In conclusion,~\eqref{e.penalized} is valid. 

\smallskip

\emph{Step 4.}
We prove that~$w^s = u^s-u \geq 0$. If~$u(x)=0$ for some~$x\in U$, then~$w^s(x) \geq 0$ by Step 1. We thus need to show that~$w^s \geq 0$ in the open  set 
$ V=\{ u> 0  \} \cap U$. Obviously,~$w^s \geq 0$ on~$\partial V$, and 
\begin{equation*}
 \nabla \cdot (\b \nabla  w^s) =f( \beta_s(u^s) -1) \leq 0 ~ \textrm{ in } V,
\end{equation*}
 which means that~$w^s$ is a weak supersolution in the set~$V$ and therefore it attains its minimum on the boundary of~$V$ and, therefore,~$w^s \geq 0$.
 
\smallskip

\emph{Step 5.}
It remains to prove~\eqref{e.penalizedgradient}. This follows by testing the subtracted equations of~$u^s$ and~$u$, and from~\eqref{e.penalized},
\begin{align}
\| \nabla  u^s-\nabla u\|^2_{L^2(U)} 
 & 
\leq \int_U (\nabla  u^s-\nabla u) \cdot \b (\nabla  u^s-\nabla u) \, dx 
\\ \notag  
& 
\leq \int_U f \left| \beta_s( u^s) -\chi_{\{u>0\}}\right| (  u^s- u)  \, dx 
\\ \notag 
&  
\leq
\| f\|_{L^1(U)} \| u-u^s\|_{L^{\infty}(U)} \leq \| f\|_{L^1(U)}  s.
\end{align}
The proof is complete.
\end{proof}

\subsection{Homogenization of  the penalized obstacle problem}

We now prove the homogenization result for the penalized problem.

\begin{proposition} \label{p.homogenization}
Let~$U$ be a bounded Lipschitz domain. Let~$q \in (d,\infty)$ and~$\delta \in (0,\infty]$. There exist constants~$C(U,d,\Lambda)<\infty$ and~$\alpha(\delta,q,d,\Lambda) \in (0,\tfrac12)$ such that the following statement is valid. Let~$f \in L^q(U)$ be a given function such that~$f > 0$ and~$g \in W^{1,2+\delta}(U)$,~$g \geq 0$. 
Let~$\varepsilon \in (0,1]$ and~$u^{s,\varepsilon}, u^s \in g+ W_0^{1,2}(U)$ solve
\begin{equation} \label{OPBepsilon}
\nabla \cdot\left( \a\left( \tfrac{\cdot}{\varepsilon}\right) \nabla u^{s,\varepsilon} \right)
=f\beta_s(u^{s,\varepsilon}), \textrm{ in } U, 
\end{equation}
and 
\begin{equation} \label{OPB}
\nabla \cdot\left( \ahom \nabla u^s \right)
=f\beta_s({u^s}), \textrm{ in } U, 
\end{equation}
respectively. Then we have that
\begin{align} \label{e.homog1}
\lefteqn{\left\| u^s-u^{s,\ep} \right\|_{L^2(U)} +\left\| \nabla u^s-\nabla u^{s,\ep} \right\|_{H^{-1}(U)} 
+\left\| \ahom \nabla u^s- \a^\ep \nabla u^{s,\ep} \right\|_{H^{-1}(U)}
} \qquad & 
\\ \notag 
& 
\leq
C \big( \mathcal{E} (\ep)\big)^\alpha \left( \| \nabla g \|_{L^{2+\delta}(U)} + \| f \|_{L^q(U)} \right) 
+ C \big( \mathcal{E} (\ep)\big)^{1-\alpha} \frac{\| \nabla g  \|_{L^2(U)}}{s} \| f \|_{L^q(U)} .
\end{align}
\end{proposition}

\begin{proof}
Let~$m = \lceil - \log_3 \ep \rceil$. The idea of the proof is to compare~$u^{s,\ep}$ to the following function, so-called \emph{two-scale expansion} of~$u^{s}$, defined as
\begin{equation}
w^{s,\ep} (x):=u^{s}(x)+\varepsilon \eta_r(x) \sum_{k=1}^{d}\partial_{x_k}(\zeta_h \ast u^{s} )(x) \phi_{m,e_k}\left(\tfrac{x}{\varepsilon}\right).
\end{equation}
Here~$\phi_{m,e}$ is the finite volume corrector corresponding to~$\cu_m$ and~$e \in \R^d$ defined in~\eqref{e.corrector.m}. 
Moreover,~$\zeta_h$ is a standard smooth mollifier supported in~$B_h$ satisfying
\begin{align*}
\int_{B_h} \zeta_h = 1, ~ \|  \nabla^k \zeta_h\| \leq C(k,d) h^{-k-d},~ k \in \N_0,
\end{align*}
and~$\eta_r\in C_c^\infty (U)$ is a smooth cutoff function satisfying
\begin{align*}
0\leq \eta_r\leq1,~ \eta_r=1 \textrm{ in } U_{2r},~ \eta_r \equiv 0 \textrm{ in } U \setminus U_r, \quad
| \nabla^k \eta_r| \leq C(k,d,U) r^{-k},
\end{align*}
and~$U_r:=\{x \in U; ~\dist(x,\partial U)>r\}$. The parameters~$h,r> 0$ will be chosen later on.

\smallskip

\emph{Step 1.} Bounds for~$u_s$ and for its convolution. First, since~$f\beta_s({u^s}) \in L^q(U)$, we have by the standard Calder\'on-Zygmund theory that there exists~$C(U,d,\Lambda) <\infty$ such that, for every~$r >0$, 
\begin{align} \label{e.reg0}
\| \nabla^2 u_s \|_{L^q(U_{r/2})} \leq C r^{-d/2} \left( r^{-1} \| \nabla g  \|_{L^2(U)} + \| f \|_{L^q(U)} \right)
.
\end{align}
Using Young's inequality for convolutions, we have that, for every~$h \in (0,r/2\wedge 1)$, 
\begin{align} \label{e.reg1}
\left\| \nabla \big( \eta_r \nabla (\zeta_h  \ast u_s) \big) \right\|_{L^\infty(U)} 
& 
\leq 
C r^{-d/2} \left( r^{-1} \| \nabla g  \|_{L^2(U)} + \| f \|_{L^q(U)} \right)
\\  \notag 
& 
\leq
C h^{-d/q} r^{-d/2} \left( r^{-1} \| \nabla g  \|_{L^2(U)} + \| f \|_{L^q(U)} \right) 
.
\end{align}
Moreover, using an elementary computation (see~\cite[Lemma 6.8]{AKMBook}) we see that 
\begin{align} \label{e.reg2}
\left\| \eta_r \nabla ( u^s - \zeta_h \ast u^s) \right\|_{L^2(U)} 
& \leq
C |U|^{1/2-1/q} h \|\nabla^2 u  \|_{L^q(U_{r/2})} 
\\ \notag 
& \leq C h^{1-d/q} r^{-d/2} \left( r^{-1} \| \nabla g  \|_{L^2(U)} + \| f \|_{L^q(U)} \right) 
.
\end{align}
Finally, by the Meyers' estimate, see Lemma \ref{l.basicregularity}, we have higher global integrability for the gradient, namely that there exist constants~$C(U,d,\Lambda)<\infty$ and~$\sigma(d,\Lambda) \in (0,\delta]$ such that 
\begin{align*} 
\left\| \nabla u^s \right\|_{L^{2+\sigma}(U)} 
\leq 
C \left( \| g  \|_{W^{1,2+\delta}(U)} + \| f \|_{L^q(U)} \right)
.
\end{align*}
which gives us, by H\"older's inequality, that
\begin{align} \label{e.reg3}
\left\| \nabla u^s \right\|_{L^{2}(U \setminus U_{2r})} 
\leq 
C r^{\frac{\sigma}{2+\sigma}} \left( \| g  \|_{W^{1,2+\delta}(U)} + \| f \|_{L^q(U)} \right)
.
\end{align}

\smallskip

\emph{Step 2.}
We derive the following estimate
\begin{equation} \label{e.betasnorm}
\| \nabla \cdot \left( \a(\tfrac{\cdot}{\varepsilon}) \nabla w^{s,\varepsilon}\right) -f\beta_s(u^s) \|_{H^{-1}(U)}
\leq 
C \big( \mathcal{E}(\ep) \big)^{\alpha} \left( \| g  \|_{W^{1,2+\delta}(U)} + \| f \|_{L^q(U)} \right)
\end{equation}
with constants~$C(U,d,\Lambda)<\infty$ and~$\alpha(\delta,d,\Lambda) \in (0,1)$. 
By a direct computation, we get
\begin{align}
\notag
\nabla w^{s,\varepsilon} 
& 
= \nabla (u^s - \zeta_h \ast u^s) + (1-\eta_r) \nabla \zeta_h \ast u^s 
\\ \notag 
&
\quad
+
\eta_r \sum_{k=1}^d \big(e_k + \nabla \phi_{e_k}(\tfrac{\cdot}{\varepsilon}) \big) \partial_{x_k} \zeta_h \ast u^s
+ 
\varepsilon \sum_{k=1}^d \phi_{e_k}(\tfrac{\cdot}{\varepsilon})\nabla \left( \eta_r \partial_{x_k}\zeta_h \ast u^s \right)
.
\end{align}
Since~$\phi_{m,e_k}$ solves the equation~$\nabla \cdot \a(\tfrac{\cdot}{\varepsilon}) \big(e_k + \nabla \phi_{m,e_k}(\tfrac{\cdot}{\varepsilon}) \big) = 0$, we obtain, by rearranging the terms, that
\begin{align}
\nabla \cdot \left(\a(\tfrac{\cdot}{\varepsilon}) \nabla w^{s,\varepsilon} \right) 
& 
=
\sum_{k=1}^d \nabla ( \eta_r \partial_{x_k} \zeta_h \ast u^s) \cdot \a(\tfrac{\cdot}{\varepsilon}) 
(e_k+ \nabla \phi_{m,e_k}(\tfrac{\cdot}{\varepsilon}) ) 
\\ \notag
& \qquad
+ 
\nabla \cdot \a(\tfrac{\cdot}{\varepsilon}) \left( (1-\eta_r) \nabla u^s 
+
\eta_r \nabla (u^s - \zeta_h \ast u^s) \right) 
\\ \notag
& \qquad 
+
\varepsilon \sum_{k=1}^d \nabla \cdot \a(\tfrac{\cdot}{\varepsilon}) \left( \phi_{m,e_k}(\tfrac{\cdot}{\varepsilon}) \nabla ( \eta_r \partial_{x_k} \zeta_h \ast u^s) \right) .
\end{align}
On the other hand, using the fact that~$\ahom$ has constant coefficients together with the equation for~$u^s$, we deduce that 
\begin{align}
\notag
\sum_{k=1}^d \nabla (\eta_r \partial_{x_k} \zeta_h \ast u^s) \cdot \ahom e_k
& 
=
\nabla \cdot (\eta_r \ahom \nabla \zeta_h \ast u^s )
\\\notag 
&
=
f\beta_s(u^s) - \nabla \cdot (\ahom \eta_r \nabla ( u^s - \zeta_h \ast u^s )) 
- \nabla \cdot ( (1-\eta_r) \ahom \nabla u^s ) 
.
\end{align}
Therefore, we get
\begin{align*} 
\nabla \cdot \left(\a(\tfrac{\cdot}{\varepsilon}) \nabla w^{s,\varepsilon } \right) 
- f \beta_s(u^s)
& = 
\sum_{k=1}^d \nabla(\eta_r \partial_{x_k}u^s) \cdot \left( \a(\tfrac{\cdot}{\varepsilon}) (e_k +\nabla \phi_{m,e_k}(\tfrac{\cdot}{\varepsilon}))-\ahom e_k\right)
\\ \notag
& \quad 
+\nabla \cdot \left( (\a(\tfrac{\cdot}{\varepsilon})-\ahom) \big( (1-\eta_r )\nabla u^s + \eta_r \nabla ( u^s - \zeta_h \ast u^s ) \big) \right) 
\\ \notag
& \quad 
+
\varepsilon \sum_{k=1}^d \nabla \cdot \left( \phi_{m,e_k}(\tfrac{\cdot}{\varepsilon}) 
\a(\tfrac{\cdot}{\varepsilon}) \nabla(\eta_r \partial_{x_k}u^s)\right) .
\end{align*}
Furthermore, we write~$\a(\tfrac{\cdot}{\varepsilon}) (e_k +\nabla \phi_{e_k}(\tfrac{\cdot}{\varepsilon}))-\ahom e_k$ by means of the finite volume skew-symmetric flux corrector~$\mathbf{S}_{ij}^{e_k,m}$ defined in~\eqref{e.fluxcorrector} so that, for~$i \in \{1,\ldots,d\}$, 
\begin{align*} 
\ep \sum_{j=1}^d \partial_j \mathbf{S}_{m,ij}^{(e_k)}(\tfrac{\cdot}{\varepsilon}) = \left( \a(\tfrac{\cdot}{\varepsilon}) (e_k +\nabla \phi_{m,e_k}(\tfrac{\cdot}{\varepsilon}))-\ahom e_k \right)_i .
\end{align*}
By skew-symmetry of~$\mathbf{S}_m^{(e_k)}$, we obtain
\begin{align*} 
\sum_{k=1}^d \nabla(\eta_r \partial_{x_k}u^s) \cdot \left( \a(\tfrac{\cdot}{\varepsilon}) (e_k +\nabla \phi_{m,e_k}(\tfrac{\cdot}{\varepsilon}))-\ahom e_k\right)
= 
\ep \sum_{k=1}^d \nabla \cdot \left( \mathbf{S}_m^{(e_k)} \nabla(\eta_r \partial_{x_k}u^s) \right)
.
\end{align*}
It follows that
\begin{align*} 
\lefteqn{
\| \nabla \cdot \left( \a(\tfrac{\cdot}{\varepsilon}) \nabla w^{s,\varepsilon} \right) -f\beta_s( u^s ) \|_{H^{-1}(U)} 
} \qquad &
\\ \notag
& \leq
\left\| (\a(\tfrac{\cdot}{\varepsilon})-\ahom) \big( (1-\eta_r )\nabla u^s + \eta_r \nabla ( u^s - \zeta_h \ast u^s ) \big) \right\|_{L^2(U)} 
\\ \notag
& 
+ \varepsilon 
\sum_{k=1}^d 
\left\| \left( \a \phi_{m,e_k} + \mathbf{S}_m^{(e_k)} \right)(\tfrac{\cdot}{\varepsilon}) 
\nabla(\eta_r \partial_{x_k}u^s) 
\right\|_{L^2(U)}.
\end{align*}
Employing the definition of~$\mathcal{E}(\ep)$ in , we finally derive
\begin{align} \notag 
\lefteqn{
\| \nabla \cdot \left( \a(\tfrac{\cdot}{\varepsilon}) \nabla w^{s,\varepsilon} \right) -f\beta_s( u^s ) \|_{H^{-1}(U)}
} \quad 
&
\\ \notag 
&
\leq 
C \left( 
\mathcal{E}(\ep) \left\| \nabla(\eta_r \partial_{x_k}u^s) \right\|_{L^\infty(U)}
+ \left\| \nabla u^s \right\|_{L^{2}(U \setminus U_{2r})} + \left\| \eta_r \nabla ( u^s - \zeta_h \ast u^s) \right\|_{L^2(U)} 
\right)
\\ \notag 
&
\leq 
C \left(  r^{-d/2-1} \big( \mathcal{E}(\ep) h^{-d/q} + h^{1-d/q} \big) + r^{\frac{\sigma}{2+\sigma}} \right) 
\left( \| g  \|_{W^{1,2+\delta}(U)} + \| f \|_{L^q(U)} \right)
.
\end{align}
Optimizing in~$h$ and~$r$, that is choosing~$h = \mathcal{E}(\ep) \wedge \frac12$ and~$r = h^{\beta}$ with the exponent~$\beta = \frac{q-d}{q} \big( 1 + \frac d2 + \frac{\sigma}{2+\sigma} \big)$, we obtain the desired inequality~\eqref{e.betasnorm} with~$\alpha = \frac{\beta \sigma}{2+\sigma}$. 

\smallskip

\emph{Step 3.} We show that 
\begin{equation} \label{e.h1norm}
\| \nabla u^{s,\varepsilon} - \nabla w^{s,\varepsilon} \|^2_{L^2(U)} 
\leq 
C 
\| \nabla \cdot \a^\ep \nabla w^{s,\varepsilon}-f \beta_s( w^{s,\varepsilon} )\|_{H^{-1}(U)}
\end{equation}
By the ellipticity assumption and the fact that~$u^{s,\varepsilon} - w^{s,\varepsilon} \in H_0^1(U)$, we get
\begin{align}
\notag
\lefteqn{
\| \nabla u^{s,\varepsilon} - \nabla w^{s,\varepsilon} \|^2_{L^2(U)} 
} 
\qquad &
\\ \notag
& 
\leq 
\int_U \left( \nabla u^{s,\varepsilon} - \nabla w^{s,\varepsilon} \right) \cdot \a\left(\tfrac{\cdot}{\varepsilon}\right)
\left( \nabla u^{s,\varepsilon} - \nabla w^{s,\varepsilon} \right)
\\ \notag
& 
\leq 
- \int_U \left( u^{s,\varepsilon} - w^{s,\varepsilon} \right) f ( \beta_s( u^{s,\varepsilon}) - \beta_s( w^{s,\varepsilon} ) )
\\ \notag
& 
+
\| \nabla \cdot \left( \a(\tfrac{\cdot}{\varepsilon}) \nabla w^{s,\varepsilon}\right) - f \beta_s( w^{s,\varepsilon} ) \|_{H^{-1}(U)}
\| \nabla u^{s,\varepsilon} - \nabla w^{s,\varepsilon} \|_{L^{2}(U)}
.
\end{align}
Since the second term on the right is negative by the monotonicity of~$\beta_s(\cdot)$, we obtain~\eqref{e.h1norm}.

\smallskip

\emph{Step 4.}
It follows from~\eqref{e.betasnorm} and~\eqref{e.h1norm}, by the triangle inequality, that
\begin{align}
\notag
\lefteqn{
\| \nabla u^{s,\varepsilon} - \nabla w^{s,\varepsilon} \|^2_{L^2(U)} 
} \qquad & 
\\ \notag
& 
\leq 
C \| \nabla \cdot \left( \a(\tfrac{\cdot}{\varepsilon}) \nabla w^{s,\varepsilon}\right) -f\beta_s(u^s) \|_{H^{-1}(U)} 
+
C \| f \beta_s( w^{s,\varepsilon}) -f \beta_s( u^{s} )\|_{H^{-1}(U)} .
\end{align}
We have, by H\"older's inequality and the Sobolev inequality, that, for every~$g \in H_0^1(U)$ with~$\|  \nabla g \|_{L^2(U)} \leq 1$,
\begin{align*} 
\int_U ( f \beta_s( w^{s,\varepsilon}) - f \beta_s( u^{s} )) g
& 
\leq 
\| f \|_{L^{d}(U)} \| \beta_s( w^{s,\varepsilon}) - \beta_s( u^{s} )\|_{L^{2}(U)} \| g \|_{L^{\frac{2d}{d-2}}(U)}
\\  \notag
& 
\leq 
\frac Cs \| f \|_{L^{d}(U)} \| w^{s,\varepsilon} - u^{s} \|_{L^{2}(U)} \| \nabla g \|_{L^{2}(U)}
\\ \notag
& 
\leq 
\frac Cs \| f \|_{L^{q}(U)} \| w^{s,\varepsilon} - u^{s} \|_{L^{2}(U)} 
.
\end{align*}
Furthermore, we have by the definitions of~$w^{s,\varepsilon}~$and~$\mathcal{E}(\ep)$, together with~\eqref{e.reg1}, that 
\begin{align*} 
\| w^{s,\varepsilon} - u^{s} \|_{L^{2}(U)} 
\leq 
C \mathcal{E}(\ep) \| \eta_r \nabla \zeta_h \ast u^{s} \|_{L^{\infty}(U)} 
\leq 
C \mathcal{E}(\ep) r^{-d/2} \| \nabla g  \|_{L^2(U)} 
.
\end{align*}
The proof is complete.
\end{proof}

We are ready to prove Theorem~\ref{t.homogenization}.

\begin{proof}[Proof of Theorem~\ref{t.homogenization}]
Combine Lemma~\ref{l.penalization} with Proposition~\ref{p.homogenization} by optimizing in~$s$. 
\end{proof}

\bigskip

\section{Large-scale $C^{1,1}$-regularity of the solution}

In this section we prove Theorem~\ref{t.C11}, which says that the solution to normalized obstacle problem satisfies $C^{1,1}$-estimate on large-scales. Let $ u^\ep \geq 0$ solve the variational inequality:
\begin{align*}
\nabla \cdot (\a^\ep \nabla u^\ep) =f \chi_{ \{ u^\ep>0\} } \textrm{ in } U.
\end{align*}
The idea of the proof is to combine boundedness of solutions close to the free boundary, provided by~\eqref{e.quadratic}, together with the regularity of the equation $\nabla \cdot (\a^\ep \nabla u^\ep) =f$, proved in Lemma~\ref{l.auxiliary} below. Indeed, if $ x_0\in \varGamma({u^\ep})$, then, by~\eqref{e.quadratic}, 
\begin{equation} \label{e.epquadratic}
\sup_{B_r(x_0)} u^\ep \leq C \| f \|_{L^\infty}  r^2
\end{equation}
provided that $ B_{2r} (x_0) \subset U$. Here the constant $C>0$ depends on $\Lambda, d$. It is known that the quadratic decay bound above does not yet imply $ C^{1,1}$-regularity of solutions, in general, if~$f$ belongs merely to $L^\infty(B_1)$. In fact, even in the case of smooth coefficients, we need to impose further assumption on~$f$ in order to obtain~$C^{1,1}$-regularity. Thus, we assume that~$f$ is H\"older continuous, which together with the homogenization then yields the large-scale $C^{1,1}$-regularity for $ u^\ep$. The terminology ``large-scale" stems from the rescaling of the problem. If we set~$R = \ep^{-1}$, we may rescale~$u^\ep$ to be the solution in a very large ball $B_{2R}$:
\begin{align*} 
\nabla \cdot (\a \nabla u) = f(\ep \cdot) \chi_{\{u >0 \}}  \quad \mbox{ in } B_{2R}. 
 \end{align*}

\smallskip

Next, recall definition of the linear space of $\ahom$-harmonic functions with controlled growth,  \eqref{e.Ak.bar},
\begin{equation*} 
\overline{\mathcal{A}}_k:=\{ u \in H^1_{loc}( \R^d)
\, : \,
-\nabla \cdot (\ahom \nabla u )=0, ~ \lim_{r\to \infty} r^{-(k+1)} \|u \| _{{\underline{L}^2(B_{r})}}=0
\}.
\end{equation*}
By the Liouville theorem,~$\overline{\mathcal{A}}_k$ consists only of~$\ahom$-harmonic polynomials of degree at most~$k$. Recall also the definition in~\eqref{e.Ak}
\begin{align*} 
{\mathcal{A}}_k
:=
\bigl\{ u \in H^1_{loc}( \R^d)
\, : \,
-\nabla \cdot (\a \nabla u )=0, ~ \lim_{r\to \infty} r^{-k-1} \| u \| _{{\underline{L}^2(B_{r})}}=0
\bigr
\}.
\end{align*}
As before, we use the notation $\psi^\ep = \psi(\tfrac \cdot\ep)$ for the rescaled members of $ \mathcal{A}_k$.

Denote the minimal scale by
\begin{align}  \label{e.minimalscale}
R_{\sigma,\alpha} := \inf \bigl\{ r \in (0,\infty) \, : \, 
\bigl(  \mathcal{E}\bigl( \tfrac{\ep}{r} \bigr)\bigr)^\alpha \leq  \sigma.
\bigr\}
\end{align}

The following proposition is essentially~\cite[Theorem 3.8]{AKMBook}. Minor modifications are needed for its proof, but the changes are almost entirely typographical. If the reader wants to go through the proof very carefully, notice that~\eqref{e.algebraic} implies that, for every $\alpha \in (0,1)$, 
\begin{align}  \label{e.haar.silly}
\int_{r}^\infty \bigl( \mathcal{E}\bigl( \tfrac{\ep}{s} \bigr)\bigr)^\alpha \, \frac{ds}{s}  
\leq 
\frac{1}{\alpha \gamma} \bigl( \mathcal{E}\bigl( \tfrac{\ep}{r} \bigr)\bigr)^\alpha
\,.
\end{align}
We omit the rest of the details of the proof. 
\begin{proposition} \label{p.regularity.homog}
Let $\gamma \in (0,1]$.  Assume that $\mathcal{E}$ satisfies~\eqref{e.algebraic} and let $R_{\sigma,\alpha}$ be defined in~\eqref{e.minimalscale}. There exist constants $\alpha,\sigma \in (0,1]$ and $C< \infty$, all depending only on $(\gamma,d,\Lambda)$, such that the following claims are valid. 
\begin{itemize}
\item 
Let $k\in\N$ and $p \in \Ahom_k$. Then there exists $\psi \in \A_k$ such that, for every $r \in [ R_{\sigma,\alpha},\infty)$,
\begin{align}  \label{e.pvspsi}
\left\| p - \psi^\ep \right\|_{\underline{L}^2 \left( B_{r} \right)}
\leq 
C \bigl( \mathcal{E}\bigl( \tfrac{\ep}{r} \bigr)\bigr)^\alpha \left\| p \right\|_{\underline{L}^2 \left( B_{r} \right)}
\end{align}

\item Let $k\in\N$ and $ \psi \in \A_k$. Then there exists $p \in \Ahom_k$ such that~\eqref{e.pvspsi} holds for every $r \in [ R_{\sigma,\alpha},\infty)$. 

\item 
If $u^\ep$ solves $-\nabla \cdot \a^\ep \nabla u^\ep = 0$ in $B_R$ with $R \geq R_{\sigma,\alpha}$, then, for every $r \in [ R_{\sigma,\alpha},R)$ and $k \in \N$, there exists $\psi \in \A_k$ such that
\begin{align}  \label{e.regularity.homog}
\left\| u^\ep - \psi^\ep \right\|_{\underline{L}^2 \left( B_{r} \right)} \leq C\left(  \frac{r}{R}  \right)^{k+1}
\left\| u^\ep  \right\|_{\underline{L}^2 \left( B_{R} \right)}
\end{align}

\end{itemize}
\end{proposition}

The next statement is the content of a much more general statement given in~\cite[Exercise 3.12]{AKMBook} with $k = 1$. Since the proof is not entirely trivial, we present the proof for the sake of completeness. 
\begin{lemma} \label{l.auxiliary}
Under the assumptions of Proposition~\ref{p.regularity.homog}, there exists a finite constant~$C(\gamma,d,\Lambda)$ such that if $R \in [ R_{\sigma,\alpha},\infty)$ and $u^\ep$ solves $\nabla\cdot \a \nabla u^\ep = f(\ep \cdot) \geq 0 $, in $B_R$, then there exists $\psi \in \A_1$ such that, for every $r \in [ R_{\sigma,\alpha},R]$, 
\begin{align*} 
\left\| u^\ep - \psi^\ep \right\|_{\underline{L}^2 \left( B_{r} \right)} 
\leq  
C\Bigl( \frac rR \Bigr)^2 \| u^\ep \|_{\underline{L}^2(B_{R})}
+
C R^2 \| f \|_{C^{0,\gamma}(B_R)} .
\end{align*}

\end{lemma}

\begin{proof}
For given $r \in [ R_{\sigma,\alpha},R]$ and $\psi_r \in \A_2$, let $w_r^\ep, ~w_r$ solve 
\begin{align*} 
\left\{
\begin{aligned}
& \nabla \cdot \a^{\ep/r} \nabla w_r^\ep = f(0) = \nabla \cdot \ahom \nabla w_r
&& 
\mbox{in } B_{1} , 
\\
& w_r^\ep = w_r = (u^\ep - \psi_r^\ep) 
&& 
\mbox{on } \partial B_{1}
\end{aligned}
\right.
\end{align*}
Lemma~\ref{l.wepandw.close} implies that, for every $\tau \in (0,\tfrac12)$, 
\begin{align*} 
\| w_r^\ep - w_r \|_{\underline{L}^2(B_{1/2})}
& 
\leq 
C \big(\mathcal{E}\bigl( \tfrac{\ep}{r} \bigr) + \tau^{\frac{\delta}{2+\delta} } \bigr) \| \nabla w_r \|_{L^{1,2+\delta} (B_{1/2})}
\\ \notag
& \quad
+ 
C \mathcal{E}\bigl( \tfrac{\ep}{r} \bigr)
\big(
\tau^{-1} \| \nabla w_r \|_{L^{\infty} (B_{\tau})} + \| \nabla^2 w_r \|_{L^{\infty} (B_{\tau})} 
\big) \,.
\end{align*}
Theorem~\ref{t.DGNM} applied to $\nabla w_r$ and $\nabla^2 w_r$, together with the Caccioppoli inequality, yields that 
\begin{align*} 
\tau^{-1} \| \nabla w_r \|_{L^{\infty} (B_{\tau})}   +  \| \nabla^2 w_r \|_{L^{\infty} (B_{\tau})}
\leq C \tau^{-2 - \frac{d}{2} } \| w_r \|_{L^{2} (B_{1/2})} .
\end{align*}
Hence, by choosing $\tau$ appropriately by means of $\mathcal{E}\bigl( \tfrac{\ep}{r} \bigr)$ and applying Caccioppoli estimate for $u^\ep - \psi^\ep$, we get that there exists a constant $\alpha(d,\Lambda) \in (0,1)$ such that 
\begin{align*} 
r^2 \| w_r^\ep - w_r  \|_{\underline{L}^2(B_{1/2})}
 \leq 
 C \bigl( \mathcal{E}\bigl( \tfrac{\ep}{r} \bigr) \bigr)^\alpha  \bigl( \| u^\ep - \psi_r^\ep\|_{\underline{L}^{2} (B_{r})} 
+ r^2 f(0) \bigr) \,.
\end{align*}
Moreover, we may compare $w_r^\ep$ and $u^\ep - \psi_r^\ep$ simply by testing as
\begin{align*} 
\| u^\ep - \psi^\ep_r  - r^2 w_r^\ep(r^{-1} \cdot)  \|_{\underline{L}^2(B_{r/2})} 
\leq C r^{2+\gamma} [f]_{C^{0,\gamma}(B_1)}.  
\end{align*}
Combining above inequalities thus yields that
\begin{align}  \label{e.harmappr.C11}
\lefteqn{
\| u^\ep - \psi^\ep_r  - r^2 w_r(r^{-1} \cdot)  \|_{\underline{L}^2(B_{r/2})} 
} \quad &
\\ 
\notag &
\leq 
 C \bigl( \mathcal{E}\bigl( \tfrac{\ep}{r} \bigr) \bigr)^\alpha  \bigl( \| u^\ep - \psi_r^\ep\|_{\underline{L}^{2} (B_{r})} 
+ r^2 f(0) \bigr) 
+
C r^{2+\gamma} [f]_{C^{0,\gamma}(B_1)}
\,.
\end{align}

Next, since $w_r - q$ is an $\ahom$-harmonic function, we deduce that, for every $s \in (0,\tfrac12 r)$, 
\begin{align} \label{e.decay.C11}
\inf_{ p \in \Ahom_2}\| r^2 w_r(r^{-1} \cdot ) - q  - p \|_{\underline{L}^2(B_{s})}
 \leq 
 C \Bigl( \frac{s}{r} \Bigr)^3 
\| r^2 w_r(r^{-1} \cdot ) - q \|_{\underline{L}^2(B_{r})} 
\,. 
\end{align}
Letting $ \tilde{p}_s \in \Ahom_2 $ be the minimizing polynomial for the infimum on the left, we see that $\tilde{p}_s$ satisfies, by the triangle inequality, that 
\begin{align}  \label{e.tildepr}
\left\| \tilde p_s \right\|_{\underline{L}^2 \left( B_{s} \right)} 
& 
\leq 
\| r^2 w_r(r^{-1} \cdot ) - q \|_{\underline{L}^2(B_{s})}
\leq 
C \Bigl( \frac{s}{r} \Bigr)^{-d/2}  \Bigl( \| u^\ep - \psi_r^\ep\|_{\underline{L}^{2} (B_{r})} 
+ r^2 f(0) \Bigr) .
\end{align}
Thus, by choosing $s = \theta r$ with $\theta(d,\Lambda) \in (0,1)$ small enough, the triangle inequality yields that 
\begin{align*} 
\| u^\ep {-} \psi_r^\ep {-} q  {-} \tilde{p}_{\theta r} \|_{\underline{L}^2(B_{\theta r})}
& 
\leq 
\frac12 \theta^2 
\| u^\ep {-} \psi_r^\ep {-} q \|_{\underline{L}^2(B_{r})} 
\\ & \quad +
 C \bigl( \mathcal{E}\bigl( \tfrac{\ep}{r} \bigr) \bigr)^\alpha  \bigl( \| u^\ep - \psi_r^\ep\|_{\underline{L}^{2} (B_{r})} 
+ r^2 f(0) \bigr) 
+
C r^{2+\gamma} [f]_{C^{0,\gamma}(B_1)}
\, .
\end{align*}
Furthermore, by Proposition~\ref{p.regularity.homog} and~\eqref{e.tildepr}, we find $\tilde{\psi}_{\theta r} \in \A_2$ such that
\begin{align*} 
\| \tilde \psi_r^\ep  - \tilde{p}_{\theta r} \|_{\underline{L}^2(B_{\theta r})} 
\leq
C \bigl( \mathcal{E}\bigl( \tfrac{\ep}{r} \bigr)\bigr)^\alpha 
\bigl( \| u^\ep - \psi_r^\ep\|_{\underline{L}^{2} (B_{r})}  + r^2 f(0) \bigr)  .
\end{align*}
Denoting hence $\psi_{\theta r} = \psi_r + \tilde{\psi}_{\theta r}$ and taking $\sigma$ so small that $ C \bigl( \mathcal{E}\bigl( \tfrac{\ep}{r} \bigr) \bigr)^\alpha \leq \frac14 \theta^2$, we  arrive at the inequality 
\begin{align*} 
\frac{\| u^\ep {-} \psi_{\theta r}^\ep {-} q  \|_{ \underline{L}^2(B_{\theta r})}}
{\theta^2 r^2}
\leq 
\frac12
\frac{ \| u^\ep {-} \psi_r^\ep {-} q \|_{\underline{L}^2(B_{r})}}{r^2}
+  
C \bigl( \mathcal{E}\bigl( \tfrac{\ep}{r} \bigr) \bigr)^\alpha f(0) 
+
C r^{\gamma} [f]_{C^{0,\gamma}(B_1)}
\,. 
\end{align*}
Integrating this against the Haar measure by using the previous inequality and using~\eqref{e.haar.silly}, we arrive after reabsorption at
\begin{align*} 
\int_{ R_{\sigma,\alpha}}^R 
\inf_{\psi \in \A_2}\| u^\ep {-} \psi^\ep {-} q  \|_{\underline{L}^2(B_{t})} \,
\frac{dt}{t^3}
\leq 
C \| u^\ep \|_{\underline{L}^2(B_{1})}
+
C \| f \|_{C^{0,\gamma}(B_1)}  \,. 
\end{align*}
Letting then $\psi_r \in \A_2$ realize the infimum above for the radius $r$, we obtain by the triangle inequality that 
\begin{align*} 
r^{-2}\| \psi_r^\ep - \psi_{2r}^\ep  \|_{\underline{L}^2(B_{r})}
 \leq 
 C 
\int_{r/2}^{3r} 
\inf_{\psi \in \A_2}\| u^\ep {-} \psi^\ep {-} q  \|_{\underline{L}^2(B_{t})} \,
\frac{dt}{t^3}
\end{align*}
Therefore, by telescoping and using the quadratic growth in the form 
$ \| \psi_r^\ep - \psi_{2r}^\ep  \|_{\underline{L}^2(B_{s})} 
\leq C (\tfrac{s}{r})^{2}
\| \psi_r^\ep - \psi_{2r}^\ep  \|_{\underline{L}^2(B_{r})}$
for every $s \in [r,1]$, and by denoting $\tilde \psi = \psi_{R_{\sigma,\alpha}}$,
we deduce that
\begin{align*} 
\| u^\ep - \tilde \psi^\ep  \|_{\underline{L}^2(B_{r})} 
\leq 
C\Bigl( \frac rR \Bigr)^2 \| u^\ep \|_{\underline{L}^2(B_{R})}
+
C R^2 \| f \|_{C^{0,\gamma}(B_R)} .
\end{align*}
Finally, we remove $\A_2\backslash \A_1$ part from $\tilde \psi$. For this, we first find, by Proposition~\ref{p.regularity.homog}, a polynomial $p_\psi = p_2  + p_1$, with $p_2 = \frac12 x \cdot \nabla^2 p_\psi x$ and $p_1 = p_\psi - p_2$, tracking down $\psi$ so that
 \begin{align*} 
\| \tilde \psi^\ep - p_\psi  \|_{\underline{L}^2(B_{r})} 
\leq 
C \bigl( \mathcal{E}\bigl( \tfrac{\ep}{r} \bigr)\bigr)^\alpha 
\|  p_\psi  \|_{\underline{L}^2(B_{r})} 
\end{align*}
Letting $\psi_2$ be given by Proposition~\ref{p.regularity.homog} tracking down $p_2$, we see that it has to have quadratic growth. In particular,
\begin{align*} 
\| \psi_2^\ep \|_{\underline{L}^2(B_{r})} 
\leq 
2\| p_2 \|_{\underline{L}^2(B_{r})}
=
2\Bigl( \frac rR \Bigr)^2 
\| p_2  \|_{\underline{L}^2(B_{R})} 
\leq
2 \Bigl( \frac rR \Bigr)^2 
\| p_\psi  \|_{\underline{L}^2(B_{R})} 
\leq
4 \Bigl( \frac rR \Bigr)^2 
\| \psi^\ep  \|_{\underline{L}^2(B_{R})} 
\,.
\end{align*}
Thus we deduce that
\begin{align*} 
\| \psi_2^\ep \|_{\underline{L}^2(B_{r})}  
\leq 
C\Bigl( \frac rR \Bigr)^2 \| u^\ep \|_{\underline{L}^2(B_{R})}
+
C R^2 \| f \|_{C^{0,\gamma}(B_R)}  .
\end{align*}
Finally, it is easy to see that $\psi = \tilde \psi - \psi_2$ belongs to $\A_1$ and this is the corrector of the statement. The proof is complete. 
\end{proof}

\begin{proof}[Proof of Theorem~\ref{t.C11}]
There are three cases. The first case is that there is $R \in [R_{\sigma,\alpha},\tfrac12]$ such that there is a free boundary point $x_0$ in $B_{2R}$, but $\varGamma(u^\ep) \cap B_R = \emptyset$. Then~\eqref{e.quadratic} implies that, for every $s \in [R,1]$, 
\begin{align*} 
\| u^\ep \|_{L^\infty(B_{s})} \leq C s^2 \| f \|_{L^\infty(B_1)}. 
\end{align*}
This, in particular, gives the bound in $L^\infty(B_{R})$.  
Applying next Lemma~\ref{l.auxiliary} in $B_R$ yields that there exists $\psi \in \A_1$ such that, for every $r \in [R_{\sigma,\alpha},R]$,  
\begin{align*} 
\| u^\ep - \psi^\ep \|_{\underline{L}^2(B_{r})}
\leq 
C\Bigl( \frac rR \Bigr)^2 \| u^\ep \|_{\underline{L}^2(B_{R})}
+
C R^2 \| f \|_{C^{0,\gamma}(B_R)} .
\end{align*}
Combining the previous two displays gives us, for $r \in [R_{\sigma,\alpha},R]$, 
\begin{align*} 
\| u^\ep - \psi^\ep \|_{\underline{L}^2(B_{r})} \leq C r^2  \| f \|_{L^\infty(B_1)}
\end{align*}
Now, since $\psi$ has linear growth, we obtain that, for every $s \in [R,1]$,
\begin{align*} 
\| \psi^\ep \|_{\underline{L}^2(B_{s})} 
 \leq 
C \frac{s}{R} \| \psi^\ep \|_{\underline{L}^2(B_{R/2})} 
\leq 
C \frac{s}{R} \Bigl( \| u^\ep \|_{\underline{L}^\infty(B_{R})}  {+}  \| u^\ep - \psi^\ep \|_{\underline{L}^2(B_{R/2})}  \Bigr)
\leq Cs^2 \| f \|_{L^\infty(B_1)},
\end{align*}
and thus~\eqref{e.C11} holds in this case. 

\smallskip

The second case is that there is a free boundary point of $u$ in~$B_{R_{\sigma,\alpha}}$. In this case~\eqref{e.C11} follows directly from~\eqref{e.quadratic}. 

\smallskip

Finally, the third case is that $\varGamma(u )\cap B_1 = \emptyset$. In the case we may directly apply Lemma~\ref{l.auxiliary} for $R = 1$ and deduce that~\eqref{e.C11} is valid with $\underline{L}^2$-norm on the left instead of $L^\infty$. However, we can apply Theorem~\ref{t.DGNM} to upgrade the integrability. The proof is complete. 
\end{proof}

\section{Homogenization  and large-scale regularity of the free boundary}

\label{s.fb}

In this section we discuss the regularity of the free boundary of the normalized heterogeneous problem. Throughout this section  $\lambda \in (0,\infty)$ is fixed.  Let~$u^\ep~$ and~$u$ be solutions  to the normalized obstacle problem, respectively, with coefficients $ \a^\ep$ and $ \ahom$,  
with~$f \equiv \lambda$, and recall that 
$\varGamma( u^\ep)$ and~$\varGamma(u)$ denote the corresponding free boundaries.

\smallskip
Our goal is to answer the following two questions;
\begin{enumerate}
  \item How close is~$\varGamma(u^\ep)$ to~$\varGamma(u)$?  
    \item What can we say about the regularity of~$\varGamma( u^\ep)$? 
\end{enumerate}

\smallskip
The discussion of the regularity of the free boundary for the homogenized obstacle problem evolves by studying the rescaled solutions around free boundary points, for $ x_0 \in \varGamma(u)$, denote  $ u_{r, x_0} =  \tfrac{u(rx +x_0)}{r^2}$. Geometrically, we zoom around a free boundary point, and investigate the limit as $ r \to 0+$. By classifying the blow-ups and showing their uniqueness at a given free boundary point, we understand the local behavior of the free boundary.  Furthermore, if we quantify the rate of the convergence to the blow-up solution as $ r \to 0+$, we automatically obtain local estimates on the regularity of the free boundary. 

\smallskip

 The argument described above is no longer applicable, when we are given merely bounded measurable coefficients. Indeed, 
let~$\ep >0$ be a fixed number, and~$x_0 \in \varGamma(u^\ep)$ be a free boundary point. Denote by
\begin{equation}
u^\ep_{r,x_0}(x):=\frac{u^\ep(rx+x_0)}{r^2}, \textrm{ where} ~r< \dist(x_0, \partial B_1).
\end{equation}
 Then~$u^\ep_{r,x_0} \geq 0$  solves the following obstacle problem
 \begin{equation}
\nabla \cdot (\a^\ep_{x_0,r} \nabla u^\ep_{r,x_0}(x)) =\lambda \chi_{\{u^\ep_{r,x_0}>0 \}},
\end{equation}
 where
  \begin{equation*}
\a_{x_0,r} :=\a\left({rx+x_0}\right)~ \textrm{ and }~ \a^\ep_{x_0,r} :=\a\left( \tfrac{rx+x_0}{\ep}\right).
\end{equation*}
 The matrix~$\a(rx+x_0)$  is uniformly elliptic with the given constants.

\smallskip
For a moment, assume that~$\a$ has continuous entries, i.e. ~$\a^\ep_{x_0,r}  \to \a^\ep({x_0})$ as~$r\to 0$.
If 
$ u^\ep_{r_j,x_0} \to u_0^\ep$ as~$r_j\to 0+$, then~$u^\ep_0 \geq 0$ is a global solution to the following obstacle problem
\begin{equation}
\nabla \cdot (\a^\ep({x_0}) \nabla u^\ep_{0}) =\lambda \chi_{\{u^\ep_{0}>0\}}.
\end{equation}
 If~$u^\ep_0$ is a half-space solution of the form
 \begin{equation} \label{ephalf-space}
 u_0^\ep(x)=\lambda \frac{(x\cdot e)_+^2}{2e\cdot \a^\ep(x_0)e},
 \end{equation}
 where~$ e \in \Rd$ is a unit vector, then~$x_0$ is called a regular free boundary point in a usual sense. If~$u^\ep_{r,x_0} \to u_0^\ep$ as~$r\to 0+$, then~$e $ is the approximate unit normal vector to~$\varGamma(u^\ep)$ at~$x_0$.
 
 \smallskip
 In a more general situation, without additional assumptions  imposed  on the matrix~$\a$, we cannot define the blow-up solution in the form~\eqref{ephalf-space}. Instead we rely on an alternative definition for regular free boundary points, based on the density of the coincidence set $\varLambda(u^\ep)$, see Definition \ref{d.regularpoint}.

 A closely related quantity in the analysis of the free boundaries is so-called minimum diameter of a set. 
 
 \begin{definition}\label{d.minimum.diameter}
Given $U \subset \R^d$, \emph{the minimum diameter}  of $U$, denoted $\mindiam(U)$, is the infimum among the distances between pairs of parallel hyperplanes enclosing~$U$.
\end{definition}

Notice that we have the following trivial lower bound for the minimum diameter by means of the density:
\begin{equation} \label{e.mdvsdensity}
2 \geq \frac{\mindiam (\varLambda \cap B_r) }{r } \geq  \frac{ |\varLambda \cap B_r|}{ d| B_r|}.
\end{equation}

 The Caffarelli's alternative in the measure form, \cite{BK,BZ2} says that if  $ \a \in VMO$, then there exists
  \begin{equation}
 \lim_{r \to 0+} \frac{| \varLambda(u) \cap B_r(x_0)|}{|B_r(x_0)|} =\frac{1}{2},
 \end{equation}
 and $ x_0$ is  a regular point,
 or otherwise
  \begin{equation}
  \lim_{r \to 0+}  \frac{ | \varLambda(u) \cap B_r(x_0)|}{|B_r(x_0)|} =0,
 \end{equation}
 and $ x_0$ is a singular point.  
Let us mention that when $ \a =\ahom$, a stronger result holds: for any $\theta >0$ small there exists an $ 0<r_0<1$ and $0<\tau <1$ such that if for some 
$ t < r_0$
  \begin{equation*}
 \frac{| \varLambda(u)  \cap B_t(x_0)|}{|B_t(x_0)|} \geq \theta,~ \textrm{ then}~~
\inf_{0<r\leq \tau t}  \frac{| \varLambda(u) \cap B_r(x_0)|}{|B_r(x_0)|} \geq \frac{1}{2}-\theta,
 \end{equation*}
 where the parameters $ r_0, \tau $ are universal, i.e. they depend on the given parameters, but not on the solution $u$.
 \smallskip

 In the next easy example, we see that the 
 Caffarelli's  alternative does not hold, when $ \a $ is discontinuous, since we may have a contact set, with density $ \frac{1}{4}$ at a free boundary point.\\
 \textit{\textbf{Example:}} Let in dimension 2, $ \a = I $ in $ (0,1) \times (0,1)\cup (-1,0) \times (-1,0)$, and 
 $ \a = 2I $ in $ (-1,0) \times (0,1) \cup (-1,0)\times (0,1) $.  Then $ u^0= x^2\chi_{\{ x>0\}}+y^2 \chi_{ \{ y>0\}} $, solves the obstacle problem 
 $ \nabla \cdot (\a \nabla u) =4 \chi_{ \{u>0\} }$, in $(-1,1)\times (-1,1)$, and the contact set $\{ u^0 \equiv 0\}$ has density $ \frac{1}{4}$ at the origin, the free boundary has a corner point. In other words, the blow-up of $ u^0$ at the origin is a homogeneous global solution, which is neither one-dimensional, nor a polynomial.

 \smallskip

 \subsection{Qualitative flatness of the free boundary.}

In this manuscript we consider the set of half-space solutions for the homogenized obstacle problem;
\begin{equation}
\mathbb{H} = \mathbb{H}_\lambda:=\left\{ \frac{\lambda}{2} \frac{(x \cdot e +t)_+^2}{e\cdot \ahom e} \, : \,  e \in \partial B_1, \, t \in \R  \right\}, 
\end{equation}
and the space of homogeneous half-space solutions, $ \mathbb{H}^0$, defined in \eqref{e.homog.blow-up}, when $ t=0$.

The following result states that the free boundary is flat around a regular free boundary point. The result is qualitative in nature since we may choose the parameter $\delta$ measuring the flatness very small, but then we do not have quantitative control of the constants appearing in the proof. Quantification will be given in the next section.

\begin{proposition} \label{p.flat.qualitative}
Let $\delta,\vartheta \in (0,1/2)$. There exist constants $r_0,\gamma \in (0,1)$ and $\sigma \in (0,\delta]$, all depending only on $(\delta,\vartheta,\lambda,d,\Lambda)$, and $\alpha(d,\Lambda) \in (0,1)$ such that the following claim is valid. Let $u^\ep$ solve the equation 
\begin{align*} 
\nabla \cdot \a^\ep \nabla u^\ep = \lambda \chi_{\{ u^\ep > 0 \} } \qquad \mbox{in } B_1\,.
\end{align*}
Suppose that $0 \in \varGamma(u^\ep)$ and that the following conditions are valid: For some $r \in (0,r_0]$ we have that
\begin{align}  \label{e.density.initial}
\inf_{ s \in (\gamma r, \gamma^{-1} r)} \frac{|\Lambda(u^\ep) \cap B_s |}{|B_s|} \geq \vartheta 
\quad \mbox{and} \quad 
\big( \mathcal{E}(\ep)\big)^\alpha \leq \sigma r^2
\end{align}
Then
\begin{equation} \label{e.l.s.halfspace}
\inf_{h \in \mathbb{H}^0} \| u^\ep -h \|_{L^\infty(B_{r})} \leq  \delta r^2 \,.
\end{equation}
\end{proposition}

To prove the above proposition, we consider solutions $u^\ep$ and $u$ of the equations 
\begin{equation} \label{e.uepu.again}
\left\{ 
\begin{aligned}
& \nabla \cdot (\a^\ep \nabla u^\ep)= \lambda \chi_{\{u^\ep >0\}}  & & \textrm{ in } B_1,
\\
& \nabla \cdot (\ahom \nabla u)= \lambda \chi_{\{u >0\}} & & \textrm{ in } B_1,
\\
& u^\ep = u & & \textrm{ on } \partial B_1 .
\end{aligned}
\right.
\end{equation}
We obtain the estimate for $ \| u^\ep -u \|_{L^\infty}$ using Theorem~\ref{t.homogenization}, Lemma~\ref{l.basicregularity} and the interpolation Lemma \ref{l.interpolation}. Indeed, there exist constants $C(\lambda,d,\Lambda)<\infty$ and $\alpha(d,\Lambda)\in (0,1)$ such that 
\begin{align} \label{e.flat.homog.error}
 \| u^\ep -u \|_{L^\infty(B_{1/2})} 
 \leq 
 C \bigl( \mathcal{E}(\ep) \bigr)^\alpha ( \| u \|_{L^2(B_1)}  + \| u^\ep  \|_{L^2(B_1)} + \lambda)
 . 
\end{align}

\smallskip

The first lemma we prove states that if the origin is a free boundary point of~$u^\ep$ and the homogenized solution is close-by pointwise, then there exists a free boundary point of the homogenized problem close to origin as well.  

\begin{lemma} \label{l.Hausdorff} (Closeness of free boundaries)
 Let $\eta \in (0,1)$. There exists a constant $\theta(\eta,\lambda,d,\Lambda) \in (0,1)$ such that the following claim is valid. Let~$u^\ep$ and~$u$ solve~\eqref{e.uepu.again}. Suppose further that $0 \in \varGamma(u^\ep)$ and that, for some~$s\in (0,1/2)$, we have that both
\begin{align}  \label{e.Hausdorff.mindist}
\mindiam (\varLambda(u^\ep) \cap B_{\theta s}) \geq \eta \theta s
\end{align}
and 
\begin{align}  \label{e.Hausdorff.radii}
 \| u^{\ep}  - u \|_{L^\infty(B_{1/2})} \leq  \theta^4 s^2
\end{align}
are valid.  Then there is a free boundary point of $u$ in $B_s$, that is, $\varGamma(u) \cap B_s \neq \emptyset$. 
\end{lemma}

\begin{proof}
Suppose that $0 \in \varGamma(u^\ep)$ and assume that~\eqref{e.Hausdorff.mindist} and~\eqref{e.Hausdorff.radii} are valid with $s\in (0,1/2)$ and with $\theta$ to be fixed. Without loss of generality, we may assume that~$u^\ep \neq u$ in~$B_1$. We claim that~$\varGamma(u) \cap B_s$ is non-empty.  Assume, on the contrary, that~$\varGamma(u) \cap B_s = \emptyset$. There are then two cases. 

\smallskip 

\emph{Case 1: $B_s \subset \varLambda(u)$.}
This is the easy case by the non-degeneracy. Indeed, we have
\begin{align*} 
c s^2 
\leq 
\sup_{B_s} u^{\ep}  
= 
\sup_{B_s} (u^{\ep}  - u) 
\leq 
\| u^{\ep}  - u \|_{L^\infty(B_{1/2})} 
\leq 
\theta^4 s^2
\,,
\end{align*}
which gives a contradiction provided that $\theta$ is small enough. 

\smallskip 

\emph{Case 2: $B_s \subset \{ u>0 \}$.} In the second case we need to use more technical argument appealing to the regularity of $u$. Since $\nabla \cdot \ahom \nabla u = \lambda$ in $B_s$, using the bound on $u^\ep$ giving 
\begin{align*} 
\| u \|_{L^\infty(B_s)} 
\leq 
\| u^\ep \|_{L^\infty(B_s)} 
+ \| u^\ep - u \|_{L^\infty(B_{1/2})} 
\leq Cs^2 + \theta^2 s^2 \leq Cs^2,
\end{align*}
we deduce by the $\ahom$-harmonicity of $x \mapsto u(x) - \frac{\lambda}{2d} x \cdot \ahom^{-1} x$ that 
\begin{align*} 
\| \nabla^2  u \|_{L^\infty(B_{s/2}) } 
+ 
s \| \nabla^3  u \|_{L^\infty(B_{s/2}) } 
\leq 
C
\,.
\end{align*}
Therefore, we obtain, with the affine function $\ell_y(x) = u(y) + \nabla u(y) \cdot {(x-y)}$ and quadratic polynomial $p_y(x) = \frac12 (x-y) \cdot \nabla^2 u(y) (x-y)$, that, for every $t \in (0,s/4]$ and $ y \in B_{s/4}$, 
\begin{align*} 
\| u - \ell_y - p_y \|_{L^\infty(B_t(y))}  \leq Ct^2 \frac{t}{s}  \,.
\end{align*}
Notice that $\ahom : \nabla^2 u(y) = \lambda$. Suppose then that $y \in \varLambda(u^\ep)$ and take $t =\| u^\ep - u \|_{L^\infty(B_{1/2})}^{1/2}$.  Then $|u(y)| \leq t^2$.
Moreover, since $u^\ep(y) = 0$, we have by~\eqref{l.BlankHao} that
\begin{align*} 
\| u \|_{L^\infty(B_t(y))} 
\leq 
\| u^\ep - u \|_{L^\infty(B_{1/2})} + \| u^\ep \|_{L^\infty(B_t(y))}
\leq 
t^2 + C t^2 
\leq 
Ct^2
\,.
\end{align*}
If $\nabla u(y) \neq 0$, we apply the above estimates at $x^* = y +  t \nabla u(y)/|\nabla u(y)|$ to obtain
\begin{align*} 
\lefteqn{
|\nabla u(y)| =t^{-1} | \ell_y(x^*)- u(y)|}\\
& \quad \quad \quad
\leq
t^{-1}\bigl(  | \ell_y(x^*) - u(x^*)| + | u(x^*)- u(y) | \bigr)
\leq 
Ct \leq C \theta (\theta s)
\,.
\end{align*}

Furthermore, since $u - \ell_0 -  p_0$ is $\ahom$-harmonic, we deduce that 
\begin{align*} 
\| \nabla u - \nabla u(0) -\nabla p_0 \|_{L^\infty(B_{\theta s})} 
\leq
\frac{C}{\theta s}
\| u - \ell_0 - p_0 \|_{L^\infty(B_{2\theta s})}
\leq 
C 
\theta (\theta s)\,.
\end{align*}
The above two displays yield the bound 
\begin{align}  \label{e.nablap.incontrol}
\| \nabla p_0 \|_{L^\infty(\varLambda(u^\ep) \cap B_{\theta s})} 
\leq 
C \theta (\theta s)\,.
\end{align}

\smallskip 

Next, the assumption~\eqref{e.Hausdorff.mindist} on the minimal distance implies that there are at least $d$ points in $\varLambda(u^\ep) \cap B_{\theta s}$, say $\{y_j\}_{j=1}^{d} \subset \varLambda(u^\ep) \cap B_{\theta s}$, such that the matrix $Z$ with columns $z_j = (\theta s)^{-1} y_j$,~$j \in \{1,\ldots,d\}$, has the modulus of the determinant bounded below by $2^{-d} \eta^d$. Indeed, the origin belongs to $U := (\theta s)^{-1} \varLambda(u^\ep) \cap B_{1}$ and $\mindiam(U)\geq \vartheta$ by~\eqref{e.Hausdorff.mindist}. Therefore, there must be a point~$z_1$ with $|z_1| \geq \frac12  \vartheta$. Let then $n \in \N$, $n < d$, and suppose that we have,  for every $k \in \{1,\ldots,n\}$, $z_k \in U$ and a hyperplane $H_k = \spn\{z_1,\ldots,z_k\}$ and projections $P_k$ onto $H_k$ and $P_k^\perp = I_d-P_k$ onto the orthogonal complement of $H_k$ such that $\alpha_{k} := |P_{k-1}^\perp z_k| \geq \frac12  \vartheta$. For $n=1$ this is valid trivially, because we may take $P_0 = 0$. Now, since $\mindiam(U)\geq \vartheta$, we thus find~$z_{n+1} \in U$ such that $\alpha_{n+1} := |P_{n}^\perp z_{n+1}| \geq \frac12  \vartheta$. Inductively we hence find $\{z_1,\ldots,z_d \}$ with aforementioned properties. 

\smallskip 

We then check that $\det(Z) \geq 2^{-d} \vartheta^d$. For this, set $Z_k = [z_1 , \ldots, z_k ]$ and compute
\begin{align*} 
\det(Z) 
& =  \big| [Z_{d-1} , z_d] \big| 
= \big| [Z_{d-1}, P_{d-1}^\perp z_d] \big| 
+  \big| [Z_{d-1}, P_{d-1} z_d] \,.
\end{align*}
The last determinant is zero since $P_{d-1} z_d \in H_{d-1}$. Furthermore, 
\begin{align*} 
\big| [Z_{d-1}, P_{d-1}^\perp z_d] \big|  = \big| [Z_{d-2}, P_{d-2}^\perp z_{d-1} , P_{d-1}^\perp z_d] \big| + \big| [Z_{d-2}, P_{d-2} z_{d-1} , P_{d-1}^\perp z_d] \big|\,,
\end{align*}
and again the last determinant is zero. Continuing using the same principle we deduce that 
\begin{align*} 
\det(Z)  = \big| [z_1,  P_{1}^\perp z_{2} , \ldots, P_{d-1}^\perp z_d] \big|\,.
\end{align*}
The projections have the property that $P_k P_m = P_k$ for every $k > m$, which implies, by the fact that $z_{k+1} = P_{m} z_{k+1}$, that
\begin{align*} 
P_k^\perp z_{k+1} \cdot P_m^\perp z_{m+1} 
=
z_{k+1} \cdot (I_d - P_k - P_m + P_k P_m) z_{m+1} 
=  
P_{m} z_{k+1} \cdot (I_d - P_m)  z_{m+1} = 0 
\,.
\end{align*}
It thus follows that 
\begin{align*} 
\det(Z) = |z_1| \prod_{k=1}^{d-1} |P_k^\perp z_{k+1}| 
 \geq 2^{-d} \vartheta^d\,, 
\end{align*}
as claimed. The determinant bound yields that~$| Z^{-1}| \leq (\det Z)^{-1} \mathrm{Adj}(Z) \leq C(d) \vartheta^{-d}$. 

\smallskip

Applying now~\eqref{e.nablap.incontrol}, we deduce that 
\begin{align*} 
| \nabla^2 u(0) z_k | = \frac1{\theta s} | \nabla p_0(y_k) | 
\leq  
\frac1{\theta s} 
\| \nabla p_0 \|_{L^\infty(\varLambda(u^\ep) \cap B_{\theta s})} 
\leq 
C \theta 
\,.
\end{align*}
It follows that 
\begin{align*} 
| \nabla^2 u(0)| =  | \nabla^2 u(0) Z Z^{-1}| \leq 
 | \nabla^2 u(0) Z|  |Z^{-1}| 
 \leq
 C \theta \vartheta^{-d} 
 \,,
\end{align*}
which then contradicts the fact that $\ahom : \nabla^2 u(0) = \lambda$ for $\theta \leq (2Cd \Lambda)^{-1} \lambda \vartheta^{d}$. 

\smallskip

The proof is now complete since we have shown that neither Case 1 nor Case 2 is possible, implying that there must be a free boundary point of $u$ in $B_s$.  
\end{proof}

\smallskip

With the aid of the previous lemma we are able to find free boundary points of~$u$ close to the origin. After a linear change of variables, we may then consider the problem $\Delta v = \lambda \chi_{\{ v>0 \}}$ and show the closeness to a half-plane solution. The following lemma is essentially a consequence of \cite[Lemma 6]{Caf80}. 
\begin{lemma}[\cite{Caf80}] \label{l.Caffarelli80}
Let $\rho, \eta \in (0,10^{-2})$. There exist constants~$\tau,s_0 \in (0,1)$ depending only on $(\rho , \eta,d)$ such that if 
~$u$ is a solution of $ \Delta u =\lambda \chi_{ \{ u>0\}}$ in $B_1$ such that~$0 \in \varGamma(u)$, $ \mindiam(\varLambda(u) \cap B_s) \geq \eta s$ for some $ s \in (0,s_0]$, then there exists $e \in \partial B_1$ such that 
\begin{equation*}
  \{ x \in B_{\tau s} \, : \, e\cdot x \leq  - \rho  \tau s    \}  \subset \varLambda(u) 
\quad \mbox{and} \quad 
  \{ x \in B_{\tau s} \, : \, e\cdot x \geq  \rho \tau s    \}  \subset \{ u > 0 \} \,.
\end{equation*}
Moreover, there exists a constant $C(\lambda,d) < \infty$ such that 
\begin{align}  \label{e.flatness.begin}
\inf_{e \in \partial B_1} \sup_{x \in B_{\tau s/12}} \Big| u(x) - \frac{\lambda}{2} (e\cdot x)_+^2  \Big|
\leq  C \rho^{1/2} (\tau s)^2
 \, .
\end{align}
\end{lemma}
\begin{proof}
The first statement is the content of \cite[Lemma 6]{Caf80}. 

\smallskip

To show~\eqref{e.flatness.begin}, denote $r =  \tau s$ and take the obvious candidate $h(x) = \frac{\lambda}{2} (x\cdot e)_+^2$ as a competitor for the infimum. Then $v = u-h$ is harmonic outside of the strip $\{ x \in B_{r} \, : \, | e\cdot x| <  \delta r \} $. Let $w$ solve $\nabla \cdot \ahom\nabla w = \nabla \cdot \ahom\nabla v$ with zero boundary values on $\partial B_r$. It is straightforward to show that $\| w \|_{L^\infty(B_r)} \leq C \rho r^2$. We may now apply, for example, a Three Ball Theorem\footnote{The Three Ball Theorem follows from Almgren's monotonicity formula, that is, $r \mapsto r \|\nabla u \|_{L^2(B_r)}^2 \| u \|_{L^2(\partial B_r)}^{-2}$ is non-decreasing.} to conclude. It states that if $g$ is a harmonic function in $B_r$, then, for every $\theta , t \in (0,1)$ such that $B_{t}(y) \subset B_{r}$, 
\begin{align*} 
\| g  \|_{L^2(B_{\theta t}(y))} 
\leq
\| g  \|_{L^2(B_{\theta^2 t}(y))}^{1/2}
\| g  \|_{L^2(B_{t}(y))}^{1/2}. 
\end{align*}
We apply this with $g = u-h-w$ and parameters $y = \frac{1}{8}e r$, $\theta = \frac13$, $t = \frac 78 r$ to get
\begin{align*} 
 \| g \|_{\underline{L}^2(B_{s/6})} 
\leq
C \| g \|_{\underline{L}^2(B_{\theta t}(y)) }   
\leq 
C \| g  \|_{\underline{L}^2(B_{\theta^2 t}(y))}^{1/2} 
\| g  \|_{\underline{L}^2(B_{7s/8}(y)) }^{1/2} .
\end{align*}
Since $g$ is harmonic and $u=h =0$ in $B_{\theta^2 t}(y)$, we obtain using the supremum bound obtained before for $w$ that 
\begin{align*} 
 \| u-h \|_{L^\infty(B_{r/12})}  
 \leq 
  \| w \|_{L^\infty(B_{r/12})} 
  +
 C \rho^{1/2} r^2 \leq  C \rho^{1/2} r^2 \,.
\end{align*}
This completes the proof of~\eqref{e.flatness.begin} since $r = \tau s$.
\end{proof}

\begin{supplementary}
\begin{lemma}[Three Spheres Theorem] \label{l.three.spheres}
Suppose that $u$ is harmonic in $B_1$. Then, for every $\theta,r\in (0,1)$, 
\begin{align}  \label{e.three.spheres}
\| u \|_{L^2(\partial B_{\theta r})}^2  
\leq
\| u \|_{L^2(\partial B_{\theta^2 r})} \| u \|_{L^2(\partial B_{r})}
\end{align}
and
\begin{align}  \label{e.three.balls}
\| u \|_{L^2(B_{\theta r})}^2  
\leq
\| u \|_{L^2(B_{\theta^2 r})} \| u \|_{L^2(B_{r})}
\,.
\end{align}
\end{lemma}
\begin{proof}
\emph{Step 1. Almgren's monotonicity formula.}  
We claim that the function 
\begin{align*} 
r \mapsto N_u(r)  
= 
\frac{ r \|\nabla u \|_{L^2(B_r)}^2} { \| u \|_{L^2(\partial B_r)}^{2}} 
=
\frac{ r^2 \|\nabla u(r \cdot) \|_{L^2(B_1)}^2} { \| u(r \cdot) \|_{L^2(\partial B_1)}^{2}} 
\end{align*}
is monotone increasing in $r$. Compute
\begin{align*} 
N_u'(r)  =
N_u(r)
\biggl( 
\frac{2}{r} 
+ 
\frac{d}{dr} \log \|\nabla u(r \cdot) \|_{L^2(B_1)}^{2} 
-
\frac{d}{dr} \log \| u(r \cdot) \|_{L^2(\partial B_1)}^{2}    
\biggr)
\,.
\end{align*}
Letting $v = r^{-1} u( r \, \cdot)$ and using the fact that $\Delta v = 0$, we obtain the identities 
\begin{align*} 
x\cdot \nabla^2 v \nabla v
& 
=
- |\nabla v|^2 + \nabla \cdot ((x \cdot \nabla v) \nabla v ) 
\quad \mbox{and} \quad
 |\nabla v|^2  = \nabla \cdot (v \nabla v)
\,.
\end{align*}
Using these and the divergence theorem we compute
\begin{align*} 
\frac{d}{dr}  \|\nabla u(r \cdot) \|_{L^2(B_1)}^{2} 
= 
\frac{2}{r} \int_{B_1} x \cdot \nabla^2 v \nabla v
=
- \frac{2}{r}  \int_{B_1} |\nabla v|^2 + \frac{2}{r}  \int_{\partial B_1} (x \cdot \nabla v)^2 
\end{align*}
and
\begin{align}  \label{e.uonsphere.der}
\frac{d}{dr}  \| u(r \cdot) \|_{L^2(\partial B_1)}^{2}    
=
\frac{2}{r} \int_{\partial B_1} (x \cdot \nabla v) v  
=
\frac{2}{r} \int_{B_1} |\nabla v|^2
\,.
\end{align}
Putting these together yields that
\begin{align*} 
N_u'(r) 
= 
\frac{2}{r} N_u(r) 
\biggl( 
\|\nabla v \|_{L^2(B_1)}^{-2} \int_{\partial B_1} (x \cdot \nabla v)^2
- 
N_u(r)
\biggr)
\,.
\end{align*}
Since $N_u(r) = \|\nabla v \|_{L^2(B_1)}^{2}  \| v \|_{L^2(\partial B_1)}^{-2}$, we obtain by H\"older's inequality and the second identity in~\eqref{e.uonsphere.der} that 
\begin{align*} 
N_u(r) =   \|\nabla v \|_{L^2(B_1)}^{-2} \| v \|_{L^2(\partial B_1)}^{-2}  \biggl( \int_{\partial B_1} (x \cdot \nabla v) v   \biggr)^2
\leq  \|\nabla v \|_{L^2(B_1)}^{-2} \int_{\partial B_1} (x \cdot \nabla v)^2
\,,
\end{align*}
and thus $N_u'(r) $ is non-negative. 

\smallskip

\emph{Step 2.}
We claim that $t \mapsto E(t) = \log \fint_{\partial B_{e^t}} u^2$ is convex. To see this, compute
\begin{align*} 
\frac{d}{dr} \log \fint_{\partial B_{r}} u^2 
=
\frac{d}{dr} \log \| u(r \cdot) \|_{L^2(\partial B_1)}^{2}     
= 
\frac2{r} N_u(r) 
\,.
\end{align*}
Therefore, 
\begin{align*} 
r \frac{d^2}{dr^2} \log \fint_{\partial B_{r}} u^2 + \frac{d}{dr} \log \fint_{\partial B_{r}} u^2 
= 
2 N_u'(r) 
\geq 
0
\,,
\end{align*}
which is then equivalent with the statement. 

\smallskip

\emph{Step 3.}
We prove~\eqref{e.three.spheres} and~\eqref{e.three.balls}. 
Apply the convexity of $E(t)$ in the form 
\begin{align*} 
E\big(\log(\theta r )\big)  
=
E\big(\tfrac12 \log(\theta^2 r ) +  \tfrac12 \log(r) \big)  \leq \tfrac12 E\big(\log(\theta^2 r )\big)  + \tfrac12 E\big(\log(r)\big) 
\end{align*}
to get~\eqref{e.three.spheres}. Finally,~\eqref{e.three.balls} follows by H\"older's inequality and~\eqref{e.three.spheres}: 
\begin{align*} 
\| u \|_{L^2(B_{\theta r})}^2  
& 
= \theta \int_0^r  \| u \|_{L^2(\partial B_{\theta t})}^2 \, dt 
\\ & 
\leq 
\theta \int_0^r  \| u \|_{L^2(\partial B_{\theta^2 t})} \| u \|_{L^2(\partial  B_{t})}  \, dt 
\\ & 
\leq 
\biggl( \theta^2 \int_0^r  \| u \|_{L^2(\partial B_{\theta^2 t})}^2 \, dt  \biggr)^{1/2} \biggl( \int_0^r \| u \|_{L^2(\partial  B_{t})}  \, dt \biggr)^{1/2} 
\\ & 
= 
\| u \|_{L^2(B_{\theta^2 r})}  \| u \|_{L^2(B_{r})} \,.
\end{align*}
The proof is complete. 
\end{proof}
\end{supplementary}

\smallskip

We are now ready to prove Proposition~\ref{p.flat.qualitative}.

\begin{proof}[Proof of Proposition~\ref{p.flat.qualitative}]
We will first apply Lemma~\ref{l.Hausdorff}. Using~\eqref{e.density.initial} with~\eqref{e.mdvsdensity}, we obtain that~\eqref{e.Hausdorff.mindist} is valid for $s = \gamma r/\theta$. Since~$\big( \mathcal{E}(\ep)\big)^\alpha  \leq \sigma r^2$, then~\eqref{e.Hausdorff.radii} is valid by~\eqref{e.flat.homog.error} for small enough $\sigma$, and we hence find $x_0 \in \varGamma(u)$ in $B_{\gamma r / \theta}$ by Lemma~\ref{l.Hausdorff}. 

\smallskip

Next, we apply~\cite[Lemma 6]{Caf98} on the density of the level strip with the parameters~$s \in [\gamma r,\gamma^{-1} r]$ and $h = s^{1/3} \| u^\ep - u \|_{L^\infty(B_{1/2})}^{1/3}$,
\begin{align*} 
|( \varLambda(u^\ep) \setminus \varLambda(u) ) \cap B_{s}(x_0) | 
& 
\leq 
| \{ 0< u < h^2 \} \cap B_{s}(x_0)| 
+ 
\frac1{h^2} \int_{\varLambda(u^\ep) \cap B_s(x_0) } u
\\ \notag & 
\leq 
C|B_s|  \Bigl( \frac{h}{s}  + \frac1{h^2} \| u^\ep - u \|_{L^\infty(B_{1/2})} \Bigr)
\\ \notag & 
\leq 
C|B_s| \bigl( \gamma^{-2} r^{-2} \big( \mathcal{E}(\ep)\big)^\alpha \bigr)^{1/3}
\\ \notag & 
\leq 
C  |B_r|(  \gamma^{-2} \sigma)^{1/3}
.
\end{align*}
Thus, by assuming that $C (  \gamma^{-2} \sigma)^{1/3} \leq 2^{-d-1}$, we deduce the lower bound for the density of $\varLambda(u)$ around the free boundary point $x_0 \in \varGamma(u)$:
\begin{align*}
 |\varLambda(u)  \cap B_s(x_0)| 
 & 
 \geq 
 |  \varLambda(u^\ep)   \cap \varLambda(u)    \cap B_s(x_0)|
 \\ &
 =
| \varLambda(u^\ep)  \cap B_{s}(x_0)| 
- 
|( \varLambda(u^\ep)\setminus \varLambda(u) ) \cap B_{s}(x_0)| 
\\ 
& 
\geq 
2^{-d-1} \theta |B_s| 
\,.
\end{align*}
Observe then that $v(x) = u\bigl (\ahom^{1/2}(x-x_0) \bigr)$  solves $\Delta v = \lambda \chi_{v>0}$ in $B_{1/\Lambda}$, $0 \in \varGamma(v)$ and that the above display and~\eqref{e.mdvsdensity} imply that $ \mindiam(\varLambda(u) \cap B_s) \geq 2^{-d-1} \Lambda^{-1/2} \theta s$ as well for every $s \in (\gamma r , \gamma^{-1} r)$. Thus, by choosing $\vartheta = 2^{-d-1} \Lambda^{-1/2}  \theta$ and $\rho =  (C^{-1}\delta \theta)^2/ (4\Lambda)$, we find by Lemma~\ref{l.Caffarelli80} constants~$s_0,\tau \in (0,1)$, both depending only on~$(\delta,\vartheta,\lambda,d,\Lambda)$, such that, for every~$s \in (0, \gamma s_0 \wedge  r]$,  
\begin{align*}
\inf_{e \in \partial B_1} \sup_{x \in B_{\tau \gamma^{-1} s/12}} \Big| v(x) - \frac{\lambda}{2} (e\cdot x)_+^2  \Big|
\leq  
\frac{\theta^2}{4 \Lambda } \delta (\tau \gamma^{-1} s/12 )^2
 \, .
\end{align*}
Finally, taking $r_0 = \gamma s_0$ so that $r \leq r_0$ implies that $\gamma^{-1} r \leq s_0$, changing variables back and taking $\gamma = \tau/(12 \Lambda^{1/2}) $ and $s = r$ above, we deduce that 
\begin{align*} 
\inf_{h \in \mathbb{H}^0} \| u^\ep -h \|_{L^\infty(B_{r})} 
\leq  
\frac{\delta}{2} r^2 
\,. 
\end{align*}
Hence~\eqref{e.l.s.halfspace} follows by the triangle inequality by demanding further that the parameter $\sigma$ is so small that $C \sigma^\alpha \leq  \frac{\delta}{2}$. The proof is complete. 
\end{proof}

\subsection{Improvement of flatness}

Next we  show that flatness implies large-scale regularity. The flatness assumption does not involve the rescaled solutions, but it helps us to keep track of rescalings, thereby obtaining large-scale regularity of the free boundary.

\smallskip

We  employ a flatness improvement argument, following J. Andersson, \cite{JASignorini}, written  for the Signorini problem. 
 The method can be compared to the De Giorgi's flatness implies regularity theorem for minimal surfaces,  \cite{Giusti}. 
For the obstacle problem the argument takes a much easier form, see Lemma \ref{l.Lapflatness}, and it can be found in  the lecture notes by J. Andersson~\cite{JAobstacle} and by J. Serra~\cite{SerraObs}.

\begin{lemma} \label{l.Lapflatness}
Let $\beta,\gamma \in (0,1)$. There exist constants~$r_0 , \delta  \in (0,1)$, both depending only on $(\beta,\gamma,\lambda,d,\Lambda)$, such that if~$u$ is a solution of~$\nabla \cdot \ahom \nabla u = \lambda \chi_{\{ u>0\}}$ in~$B_{1}$, the origin is a free boundary point of $u$, that is, $0 \in \varGamma(u)$, and that 
\begin{align}  \label{e.Lapflatness.cond}
\inf_{h \in \mathbb{H}}   \| u - h \|_{\underline{L}^2(B_{1/2})} \leq \delta ,
\end{align}
then there exists $ h \in \mathbb{H}^0 $ such that, for every $r \in (0,r_0]$, 
\begin{equation} \label{e.rgamma}
 \| u -h \|_{L^\infty(B_r)} \leq  \gamma r^{2+\beta}  \inf_{h \in \mathbb{H}} \| u-h \|_{\underline{L}^2(B_{1/2})} . 
\end{equation}
\end{lemma}


\begin{proof} 
Without loss of generality, assume that $ \ahom =I$, and hence $ u $ solves $ \Delta u = \lambda \chi_{\{ u>0\}}$.  Assume that $ 0\in \varGamma(u)$. Fix $\beta,\gamma  \in (0,1)$.
\smallskip

Our initial goal is to show that there exist small constants~$\delta,s \in (0,1)$, both depending only on $(\beta,\gamma,\lambda, d)$, such that if 
 \begin{align*}
 \inf_{(e,t) \in \partial B_1 \times \R}  \| u - \tfrac{\lambda}{2}  (e \cdot x+t)_+^2 \|_{L^2(B_{1/2})} \leq \delta,
 \end{align*}
  then 
\begin{equation}\label{e.s.tau.t.init}
\inf_{(e,t) \in \partial B_1 \times \R} 
\| u -\tfrac{\lambda}{2} (e \cdot x+t)_+^2 \|_{L^\infty(B_{s})} 
\leq \gamma s^{2+\beta} 
\inf_{(e,t) \in \partial B_1 \times \R} 
\| u -\tfrac{\lambda}{2} (e \cdot x+t)_+^2 \|_{\underline{L}^2(B_{1/2})}.
\end{equation}
We argue by contradiction, based on the following elementary argument in Step ~1. 

\smallskip

\emph{Step 1.}
We show that, for every $\gamma \in (0,1)$,  there exists a constant $c(\gamma,d) \in (0,1/2]$ such that if~$w \in H^1_{\mathrm{loc}}(B_1)$ is the weak solution of the equation
\begin{equation*}
\Delta w =0 \textrm{ in } B_1^+:=B_1 \cap \{ x_d >0\}, \textrm{ and } w \equiv 0  \textrm{ in } B_1^-:=B_1 \cap \{ x_d <0\} .
\end{equation*}
and if $w$ satisfies the conditions $\int_{B_c} w(x) (x_d)_+ \, dx = 0$ and $\int_{B_c} w(x) x_k (x_d)_+\, dx = 0$ for $k \in \{1,\ldots,d-1 \}$, then, for every $s \in (0,c]$, 
\begin{align} \label{e.w.decay}
\left\| w \right\|_{L^\infty \left( B_{s} \right)} \leq \frac{\gamma}{2} s^{2+\beta} \left\| w \right\|_{\underline{L}^2 \left( B_{1} \right)} . 
\end{align}
The proof is based on the fact that we find by odd reflection a harmonic function $w'$ in $B_1$, which coincides with $w$ in $B_1 \cap \{ x_d >0\}$. By harmonicity, there exists a radius $r(d) \in (0,1)$ such that  $w' = \sum_{k=0}^\infty p_k$ in $B_r$, where $p_k$ is a harmonic homogeneous polynomial of degree $k$. Now, the boundary condition $w' = 0$ on $B_1 \cap \{ x_d  = 0\}$ dictates that $p_0 = 0$, $p_1(x) = a_0 x_d$ and $p_2(x) =  \sum_{k=1}^{d-1} a_k x_k x_d$. Then the orthogonality conditions imply that $p_1 = 0$ and $p_2 = 0$, and that, for every $s \in (0,c]$,~$\| w \|_{L^\infty(B_s)} \leq C s^3 \| w\|_{L^2(B_1)}$. We may then choose $c$ small enough so that~$c^{1-\beta} 2 C \leq \gamma$ and obtain~\eqref{e.w.decay}. 

\smallskip

\emph{Step 2.}  Fix $\gamma \in (0,1)$ and let $c(\gamma,d)$ be as in Step 2. We now prove~\eqref{e.s.tau.t} via contradiction. 
We assume, on the contrary, that there exist a sequence~$\{u^j\}_j$ solving $\Delta u^j = \lambda \chi_{\{u^j > 0\}}$ with $0 \in \varGamma({u_j})$ and another sequence~$\{h^j\}_j$ with~$h^{j}(x) = \tfrac{\lambda}{2} (e^{j} \cdot x+t_{j})_+^2$ such that~$\delta_j := \| u^j - h^j \|_{L^2(B_1)} \to 0$ as~$j\to \infty$ while 
\begin{align} \label{e.Lapflatness.step3.contra}
\inf_{(e,t) \in \partial B_1 \times \R} 
\| u^j  -\tfrac{\lambda}{2} (e \cdot x+t)_+^2 \|_{\underline{L}^2(B_{c})} 
> 
 \gamma c^{2+\beta} \| u^j - h^j \|_{\underline{L}^2 \left( B_{1} \right)} 
. 
\end{align}
After rotation, assume, up to a relabeled subsequence, that $e^j \to e_d$ as $j \to \infty$. Set 
\begin{equation}\label{e.vj}
v^j =\frac{u^j -{h}^j }{\| u^j -{h}^j \|_{L^2(B_1)} }. 
\end{equation}
Since both $u^j$ and $h^j$ are nonnegative minimizers of the energy $\int_{B_1}(\tfrac12 |\nabla w|^2 + \lambda w)$, one actually obtains using admissible competitors that $v^j \Delta v^j \geq 0$ and $\Delta (v^j)^2 \geq 0$ weakly in~$B_1$. The latter implies that $\| v^j \|_{L^\infty(B_{1/2})} \leq C(d)$. Using the first one, it is possible to show, up to passing to a subsequence, that there exists~$v \in H_{\mathrm{loc}}^1(B_1)$ such that~$v^j \to v$ in~$H_{\mathrm{loc}}^1(B_1)$ as~$j \to \infty$.  Now, since~$0 \in \varGamma({u^j})$, we have that $|v^j(0)| = \delta_j^{-1} (t^j)_+^2 \leq C$. The non-degeneracy, on the other hand, implies that $(t^j)_-^2 \leq C \delta_j$, and thus $t_j \to 0$ as $j \to \infty$. Applying Lemma~\ref{l.BlankHao}, we may show that  $v^j \to v = 0$ in $B_1 \cap \{ x_d < 0\}$ (using~\eqref{e.nondeg}) and that $\Delta v = 0$ in $B_1 \cap \{ x_d > 0\}$ (using~\eqref{e.quadratic}). Furthermore, for given $s\in (0,1)$, letting $h^{j,s} = \tfrac{\lambda}{2} (e^{s,j} \cdot x+t^{s,j})_+^2$ be the minimizing half-space solution for $u^j$ in $B_s$ and taking variations with respect to~$t$ and~$e$, we see that, for every $e' \in (e^{s,j})^\perp$, 
   \begin{equation*} 
 \int_{B_s}  (h^{j,s})^{1/2} \left(  u^j(x)-  h^{j,s} \right)  dx =  0 = 
  \int_{B_s}  (h^{j,s})^{1/2} (e' \cdot x) \left(  u^j(x)-  h^{j,s} \right)  dx \,.
 \end{equation*}
Since $t_j \to 0$  and $e^j \to e_d$ as $j \to \infty$, the above display yields, after passing to the limit, that $\int_{B_c} v(x) (x_d)_+ \, dx = 0$ and $\int_{B_c} v(x) x_k (x_d)_+\, dx = 0$ for $k \in \{1,\ldots,d-1 \}$. Thus we are in position to use the result of Step 2 and obtain a contradiction with~\eqref{e.Lapflatness.step3.contra} for large enough $j$ by the strong convergence in $L^2$. 

\smallskip

\textit{Step 3.} Iterating~\eqref{e.s.tau.t} and using the sup bound for $(u-h)^2$ gives us that, for every $s \in (0,c]$, 
\begin{equation}\label{e.s.tau.t}
\inf_{(e,t) \in \partial B_1 \times \R} 
\| u -\tfrac{\lambda}{2} (e \cdot x+t)_+^2 \|_{L^\infty(B_{s})} 
\leq \gamma s^{2+\beta} 
\inf_{(e,t) \in \partial B_1 \times \R} 
\| u -\tfrac{\lambda}{2} (e \cdot x+t)_+^2 \|_{\underline{L}^2(B_{1/2})}.
\end{equation}
Let $e^s$ and $t^s$ realize the minimum on the right and denote $h^s = \tfrac{\lambda}{2} (e^s \cdot x+t^s)_+^2$. Then, by the triangle inequality and simple manipulations,
\begin{align*} 
s \bigl(s | e^{2s} - e^{s}|  + | t^{2s} - t^{s}|  \bigr) 
\leq 
C \gamma s^{2+\beta} 
\inf_{(e,t) \in \partial B_1 \times \R} 
\| u -\tfrac{\lambda}{2} (e \cdot x+t)_+^2 \|_{\underline{L}^2(B_{1/2})}
\end{align*}
Thus both $\{e^s\}_s$ and  $\{t^s\}_s$ are Cauchy sequences, and $e^s \to e$ and $t^s \to t$ as $s \to 0$. Since $u(0) = 0$, we must have that $t = 0$, and therefore, from the previous display, $s |e^s - e|  + |t^s| \leq C \gamma s^{2+\beta} 
\inf_{(e,t) \in \partial B_1 \times \R} 
\| u -\tfrac{\lambda}{2} (e \cdot x+t)_+^2 \|_{\underline{L}^2(B_{1/2})}$. Defining hence $h = \tfrac{\lambda}{2} (e \cdot x)_+^2$ and taking $\gamma$ smaller, we conclude with~\eqref{e.rgamma}. The proof is complete.
\end{proof}

 The quantitative estimate above, extends the  regularity of the free boundary result to a neighborhood of the origin, i.e let $ y_0 \in \varGamma(u) \cap B_{1/2}$, then there exists $ h^{y_0}$, such that 
\begin{equation*}
\| u_{r,y_0} - h^{y_0} \|_{L^2(B_{1/2})} \leq C r^\beta \| u-h \|_{L^2(B_1)},
\end{equation*}
and the corresponding unit normal vector $ e^{y_0}$ is H\"older continuous;
\begin{equation*}
| e^{y_0} -e^0 | \leq C| y_0 |^\beta,
\end{equation*}
hence $ \varGamma(u) \cap B_{1/2}$ is a $ C^{1,\beta}$-surface.

\begin{lemma} \label{l.flatnessdecay}
Fix $\beta \in (0,1)$. There exist constants~$\rho , \omega  \in (0,\tfrac14)$, both depending only on $(\beta,\lambda,d,\Lambda)$, such that if
$u^\ep$ and $u$ solve 
\begin{equation} \label{e.flatnessdecay.eqs}
\left\{ 
\begin{aligned}
& \nabla \cdot (\a^\ep \nabla u^\ep)= \lambda \chi_{\{u^\ep >0\}}  & & \textrm{ in } B_1,
\\
& \nabla \cdot (\ahom \nabla u)= \lambda \chi_{\{u >0\}} & & \textrm{ in } B_1,
\\
& u^\ep = u & & \textrm{ on } \partial B_1,
\end{aligned}
\right.
\end{equation}
and if $0 \in \varGamma(u^\ep)$ and 
\begin{align} \label{e.flatnessdecay.ass2}
\inf_{h \in \mathbb{H}}\| u^\ep  -h \|_{L^\infty(B_{1/2})}
+
\| u^\ep  - u \|_{L^\infty(B_{1/2})} 
\leq 
\omega
\,,
\end{align}
then 
\begin{equation} \label{e.flatnessdecay}
\inf_{h \in \mathbb{H}}  \| u^\ep -h \|_{L^\infty(B_\rho)} \leq  \rho^{2+\beta}  \inf_{h \in \mathbb{H}} \| u^\ep-h \|_{L^\infty(B_{1})} + 2 \| u^\ep - u \|_{L^\infty(B_{1/2})} 
\,.
\end{equation}
\end{lemma}

\begin{proof}
Fix $\beta \in (0,1)$ and let $\gamma = 2^{-2-\beta}$. Let $r_0,\delta$ be the parameters from Lemma~\ref{l.Lapflatness} corresponding $(\beta,\gamma)$. Fix $\rho = \frac12 r_0$. Let $\theta \in (0,1)$ be a small parameter to be fixed and set~$s = \theta^{-1} \omega^{1/2}$. We assume that~$\omega$ is so small that $s \leq \rho$.  
Suppose that $0 \in \varGamma(u^\ep)$ and that $h(x) = \tfrac{\lambda}{2} \frac{ (e \cdot x+t)_+^2}{e \cdot \bar \a e}$ is the member of~$\mathbb{H}$ realizing the infimum in~\eqref{e.flatnessdecay.ass2}.

\smallskip

We first show that $|t|$ in the definition of $h$ is small.  On the one hand, by non-degeneracy~\eqref{e.nondeg}, we get, for every $r \in (0,1/2)$, that 
\begin{align*} 
c r^2 \leq \| u^\ep  \|_{L^\infty(B_{r})} 
\leq 
\| h  \|_{L^\infty(B_{r})} + \| u^\ep - h  \|_{L^\infty(B_{r})} 
\leq 
\| h  \|_{L^\infty(B_{r})} + \omega
\,.
\end{align*}
Therefore, taking $r = 2 (\omega/c)^{1/2}$, we obtain that 
\begin{align*} 
\| h  \|_{L^\infty(B_{r})} > 0
\quad \implies \quad  
t > - r = - 2 (\omega/c)^{1/2} \geq - C \theta s \,. 
\end{align*}
On the other hand,  since $u^\ep(0) = 0$, 
\begin{align*} 
t_+^2 = | u^\ep(0)  -h(0)| \leq \| u^\ep  -h \|_{L^\infty(B_{1/2})} \leq \omega = \theta^2 s^2.
\end{align*} 
Thus, we obtain that $|t| < C\theta s$.  Furthermore, by the triangle inequality, we deduce with $h_0 = \tfrac{\lambda}{2} \frac{ (e \cdot x)_+^2}{e \cdot \bar \a e}$ and 
\begin{align*} 
\| u  - h_0 \|_{L^\infty(B_{\rho})}  \leq \omega + C  (\rho+|t|) |t|  \leq C\theta (\rho + s) s  \leq C\theta  \rho^2 \, .
\end{align*}
For small enough $\theta(\lambda,d,\Lambda)$ this implies that there exists a point~$x_0$ in~$\varGamma(u) \cap B_\rho$.  Indeed, it is clear that $B_\rho \subset \varLambda(u)$ is impossible. On the other hand, if $B_\rho$ would belong to $\{ u>0 \}$, then $u$ would solve $\nabla \cdot \ahom \nabla u = \lambda$ in $U_\rho := B_{\rho/2}(- e \rho/2)$ with $\| u \|_{L^\infty(U_\rho)} \leq C \theta \rho^2$. Then the function 
\begin{align*} 
v(x) = u(x) + \rho^2 \frac{\lambda}{8} \Bigl( (2 \rho x + e) \cdot \ahom^{-1} (2 \rho x + e) -  e \cdot \ahom^{-1} e \Bigr) 
\end{align*}  
would be $\ahom$-harmonic in $U_\rho$ and have boundary values less than $C \theta \rho^2$. Therefore, by the maximum principle, $\| v \|_{L^\infty(U_\rho)} \leq C \theta \rho^2$. For small enough $\theta$ this would then violate the bound $\| u \|_{L^\infty(U_\rho)} \leq C \theta \rho^2$. In conclusion, there exists a point~$x_0$ in~$\varGamma(u) \cap B_\rho$.

\smallskip

We now obtain by Lemma~\ref{l.Lapflatness} that
\begin{align*} 
\inf_{h \in \mathbb{H}} \| u  - h \|_{L^\infty(B_{\rho})} 
\leq 
\inf_{h \in \mathbb{H}} \| u  - h \|_{L^\infty(B_{2\rho}(x_0))}
\leq 
\rho^{2+\beta} \inf_{h \in \mathbb{H}} \| u  - h \|_{L^\infty(B_{1/2})},
\end{align*}
and, once again by the triangle inequality, 
\begin{align*} 
\inf_{h \in \mathbb{H}} \| u^\ep  - h \|_{L^\infty(B_{\rho})} 
\leq 
\rho^{2+\beta} \inf_{h \in \mathbb{H}} \| u^\ep  - h \|_{L^\infty(B_{1})} 
+ 
2 \| u^\ep  - u \|_{L^\infty(B_{1/2})}
\,. 
\end{align*}
The proof is complete.
\end{proof}

\begin{proof}[Proof of Theorem~\ref{t.fb.regularity}]
Fix $\beta,\vartheta \in (0,\frac12)$ and let $\omega,\rho  \in (0,1)$ be as in Lemma~\ref{l.flatnessdecay} corresponding $2\beta$. Throughout the proof, for given $t \in (0,1]$, we consider solutions $u_t^\ep = t^{-2} u^\ep(t \cdot )$ and $u_t$ of the equations
 \begin{equation*} 
\left\{ 
\begin{aligned}
& \nabla \cdot (\a^{\ep/t} \nabla u_t^\ep)= \lambda \chi_{\{u_t^\ep >0\}}  & & \textrm{ in } B_1,
\\
& \nabla \cdot (\ahom \nabla u_t)= \lambda \chi_{\{u_t >0\}} & & \textrm{ in } B_1,
\\
& u_t = u_t^\ep & & \textrm{ on } \partial B_1.
\end{aligned}
\right.
\end{equation*}
Since $0 \in \varGamma(u_t^\ep)$, we have by Lemmas~\ref{l.basicregularity} and~\ref{l.BlankHao} that 
\begin{align*} 
\left\| \nabla u_t^\ep  \right\|_{L^{2+\alpha} \left(B_1\right)}  
\leq C \Bigl( \left\| u_t^\ep  \right\|_{L^{2} \left(B_2 \right)}  + \lambda\Bigr)  \leq C. 
\end{align*}
Thus Theorem~\ref{t.homogenization} and Lemmas~\ref{l.interpolation} and~\ref{l.basicregularity} yield that, for every $t \in (0,1)$, 
\begin{align*} 
\|  u_t^{\varepsilon} - u_t \|_{L^\infty(B_{1/2})} 
\leq 
C 
\big( \mathcal{E} \bigl( \tfrac{\ep}{t} \bigr) \big)^{\alpha} 
\,. 
\end{align*}
Furthermore, since we assume~\eqref{e.thm.regular}, we may apply  Proposition~\ref{p.flat.qualitative} with $\delta = \frac12 \omega$ and obtain that 
\begin{align*} 
\inf_{h \in \mathbb{H}^0} \| u^\ep -h \|_{L^\infty(B_{r_0})} \leq  \frac{\omega}{2} r_0^2 
\quad \mbox{and} \quad
C \big( \mathcal{E}(\ep)\big)^\alpha \leq  \frac12 \bigl( \omega \wedge \sigma \bigr) r_0^2 
\,. 
\end{align*}
Therefore, combining the above two displays, yields that 
\begin{align*} 
\inf_{h \in \mathbb{H}} \| u_{r_0}^\ep -h_{r_0} \|_{L^\infty(B_{1})} 
+ 
\|  u_{r_0}^{\varepsilon} - u_{r_0} \|_{L^\infty(B_{1})} 
\leq 
\omega
\,.
\end{align*}
Lemma~\ref{l.flatnessdecay} then gives us
 \begin{align} \label{e.flatnessdecay.applied}
\inf_{h \in \mathbb{H}}  \| u_{\rho r_0}^\ep -h \|_{L^\infty(B_\rho)} 
& 
\leq  
\rho^{2\beta}  \inf_{h \in \mathbb{H}} \| u_{r_0}^\ep-h \|_{L^\infty(B_{1})} 
+ 2 \rho^{-2} \| u_{r_0}^\ep - u_{r_0} \|_{L^\infty(B_{1/2})} 
\\ \notag
& 
\leq 
 \rho^{2\beta}  \inf_{h \in \mathbb{H}} \| u_r^\ep-h \|_{L^\infty(B_{1})}  
+ 
C \big( \mathcal{E} \bigl( \tfrac{\ep}{r_0} \bigr) \big)^{\alpha} 
\,.
\end{align}
Setting $r_k = \rho^k r_0$, $k\in\N$, we can iterate this as long as $m \in\N$ is such that
\begin{align*} 
4 C \big( \mathcal{E} \bigl( \tfrac{\ep}{r_m} \bigr) \big)^{\alpha}  \leq \omega \,.
\end{align*}
This condition is guaranteed by the minimal scale condition as long as $r_m \geq R_{\sigma,\alpha}$, where $\sigma \leq (2C)^{-1} \omega$. Indeed, 
we inductively assume that, for every $k \in \{1,\ldots,n \}$, we have that
\begin{align*} 
\inf_{h \in \mathbb{H}}  \| u_{r_k}^\ep -h \|_{L^\infty(B_{1/2})} 
& 
\leq 
r_k^{2\beta} 
\inf_{h \in \mathbb{H}} \| u_r^\ep-h \|_{L^\infty(B_{1})}  
+
C \sum_{j=0}^{k-1} \big( \tfrac{r_k}{r_j}\bigr)^{2\beta}   \big( \mathcal{E} \bigl( \tfrac{\ep}{r_j} \bigr) \big)^{\alpha} .
\end{align*}
Thus we get, by taking $\sigma$ small enough,
\begin{align*} 
\inf_{h \in \mathbb{H}}  \| u_{r_k}^\ep -h \|_{L^\infty(B_{1/2})} 
& 
\leq 
r_k^{2\beta} 
\inf_{h \in \mathbb{H}} \| u_r^\ep-h \|_{L^\infty(B_{1})}  
+
C \sum_{j=0}^{k-1} \big( \tfrac{r_k}{r_j}\bigr)^{2\beta}   \big( \mathcal{E} \bigl( \tfrac{\ep}{r_j} \bigr) \big)^{\alpha}  
\leq \frac{\omega}{2}
\end{align*}
and, as already observed before,
\begin{align*} 
\|  u_{r_k}^{\varepsilon} - u_{r_k} \|_{L^\infty(B_{1/2})}  
\leq 
C \big( \mathcal{E} \bigl( \tfrac{\ep}{r_k} \bigr) \big)^{\alpha} \leq \frac{\omega}{2}
\,,
\end{align*}
and hence we may take an induction step applying Lemma~\ref{l.flatnessdecay}. In conclusion, 
we have so far shown, after relabelling constants, that there exist constants $\alpha(d,\Lambda) \in (0,1)$ and $C(\beta,\vartheta,\lambda,d,\Lambda)<\infty$ such that, for every $r\in [R_{\sigma,\alpha} ,1]$, 
\begin{align}  \label{e.fb.regularity.almost}
\inf_{h \in \mathbb{H}} \| u_r^\ep-h \|_{L^\infty(B_{1/2})} 
\leq 
Cr^{2\beta} + C \big( \mathcal{E} \bigl( \tfrac{\ep}{r} \bigr) \big)^{2\alpha}
\,.
\end{align}

\smallskip

We then move towards showing the uniform estimate. 
Letting $h_r(x) = \frac{\lambda}{2} \frac{ ( e_r \cdot x + t_r)_+^2 }{e_r \cdot \bar \a e_r}$ be the realizing the infimum in~\eqref{e.fb.regularity.almost}, we have that 
\begin{align*} 
 |e_{2s} - e_s| + s^{-1} |t_{2s} - t_s| 
& 
\leq
C \inf_{h \in \mathbb{H}} \| u_s^\ep-h \|_{L^\infty(B_{1/2})} + C \inf_{h \in \mathbb{H}} \| u_{2s}^\ep-h \|_{L^\infty(B_{1/2})}  
\\ &  
\leq 
Cs^{2\beta} + C \big( \mathcal{E} \bigl( \tfrac{\ep}{s} \bigr) \big)^{2\alpha} 
 \,. 
\end{align*}
This implies that by the condition~\eqref{e.algebraic} that, for every $s \in [2^k r,2^{k+1}r]$,  
\begin{align*} 
s^{-1} |t_s - t_r| + |e_s - e_r| \leq C r^{2\beta} \sum_{j=0}^k 2^{j2\beta} + C \sum_{j=0}^k \big( \mathcal{E} \bigl( \tfrac{\ep}{2^j r} \bigr) \big)^{2\alpha} \leq C s^{2\beta} + C \big( \mathcal{E} \bigl( \tfrac{\ep}{r} \bigr) \big)^{2\alpha}
\,.
\end{align*}
Furthermore, similarly to the proof of Lemma~\ref{l.flatnessdecay}, we also obtain that 
\begin{align*} 
|t_r| \leq Cr^{1+\beta} + C r \big( \mathcal{E} \bigl( \tfrac{\ep}{r} \bigr) \big)^{\alpha} \,.
\end{align*}
Using these we obtain that, for $h^{(r)} = \frac{\lambda}{2} \frac{ ( e_r \cdot x)_+^2 }{e_r \cdot \bar \a e_r}$ 
\begin{align*} 
\| h_{r} - h_s\|_{L^\infty(B_s)} \leq  C s^2 \Bigl(s^{2\beta} +  r^{\beta}  + \big( \mathcal{E} \bigl( \tfrac{\ep}{s} \bigr) \big)^{2\alpha} + \big( \mathcal{E} \bigl( \tfrac{\ep}{r} \bigr) \big)^{\alpha} \Bigr) 
\,.
\end{align*}
Combining this with~\eqref{e.fb.regularity.almost} then concludes the proof. 
\end{proof}

We finally remark that if the minimal scale is independent of the reference point, that is, for every $x \in \R^d$ we have
\begin{align}  \label{e.uniform.E}
\mathcal{E}( \ep, x) \leq \mathcal{E}( \ep) \qquad \forall x \in \R^d,
\end{align}
then we will obtain large-scale regularity of the free boundary. Notice that~\eqref{e.uniform.E} is valid, for example, in the case of periodic or almost periodic homogenization.

\begin{theorem}\label{t.flatnessimpliesregularity}
Assume that~\eqref{e.uniform.E} is valid.  Let $ u^\ep $ be a solution to the heterogenous obstacle problem, with  $ f \equiv \lambda >0$. Assume that
$ 0 \in \varGamma(u^\ep)$, and $ \| u^\ep - h \|_{L^2(B_2)} =\delta$ is small enough, then $ \varGamma(u^\ep) \cap B_{1/2}$ is large-scale $C^{1,\beta}$-regular, in the sense that every $ x_0 \in  \varGamma(u^\ep) \cap B_{1/2} $ is a  large-scale regular free boundary point, (see Definition \ref{d.regularpoint}), and  \eqref{e.ls.fb.c1beta} holds at $x_0$.  
\end{theorem}

\begin{proof}
The Lebesgue measure of the set $ \varLambda(u^\ep_r) \cap B_1 $ can be estimated  from below, employing Lemma \ref{l.flatnessdecay}. Without loss of generality, let us assume that the half-space solution $ h$ in the 
Lemma \ref{l.flatnessdecay}
depends on direction $ e=e_d$. Hence
 \begin{equation}
  \| u^\ep_{r}  \|_{L^\infty(B_1 \cap \{ x_d<0\} )}  \leq C r^\beta \delta +C    \mathcal{E}(\tfrac{\ep}{r} )^{\alpha} .
\end{equation}
Now let $ x^0  \in B_1 \cap \{ x_d<0\} $ be such that $ u_r^\ep (x^0)>0$, and let $2 \tau :=- x_d^0=-x^0 \cdot e_d>0$, then 
by the non-degeneracy of $ u_r^\ep$,
\begin{align*}
c \tau^2 \leq \sup_{B_\tau(x_0)} u^\ep_r
 \leq C r^{\beta} \delta +C    \mathcal{E}(\tfrac{\ep}{r} )^{\alpha}.
\end{align*}
 Therefore
we obtain
\begin{align*}
\tau^{2} \leq 
C r^\beta \delta +C (\mathcal{E} (\tfrac{\ep}{r} ))^{  \alpha}.
\end{align*}
Hence, 
\begin{align*}
| \{ u_r^\ep >0\} \cap \{ x_d <0 \} | \leq c_d \tau 
\end{align*}
is small, and
\begin{align*}
\frac{| \varLambda( u^\ep_r) \cap B_1| }{|B_1|} \geq 
\frac{1}{2}-C\tau >\vartheta>0,
\end{align*}
if $ r\geq R_{\sigma,\alpha}$.
Applying Theorem \ref{t.fb.regularity}, we obtain  \eqref{e.ls.fb.c1beta}.

Now if $ x_0 \in  \varGamma(u^\ep) \cap B_{1/2} $, then we can apply Lemma \ref{l.flatnessdecay} and Theorem \ref{t.fb.regularity} for the rescaled solutions $ u^\ep_{r, x_0}$.
\end{proof}


\bigskip

\bigskip

\subsection*{Acknowledgments}
This project has received funding from the Academy of Fin- land and the European Research Council (ERC) under the European Union's Horizon 2020 research and innovation programme (grant agreement No 818437).

\bigskip 

\appendix
\section{Examples in dimension one}
\label{s.appendix}
In this appendix we briefly discuss homogenization of the obstacle problem and its free boundary  in dimension one.
We give two important one-dimensional examples, which clarify the main assumptions and the setting of our problem.

\smallskip

\subsection{Homogenization of the normalized obstacle problem and its free boundary in 1D}

Let the coefficient field~$\a$  be a continuous function taking values in~$[1,2]$. Let~$u_\varepsilon \geq 0$ be the solution of the following  
one-dimensional obstacle problem,  with oscillating coefficients;
\begin{align} \notag 
\left\{
\begin{aligned}
 & 
(a(\tfrac{\cdot}{\varepsilon})u^\prime_\varepsilon)^\prime =\chi_{\{u_\varepsilon>0\}}
 \quad  \mbox{in } (0,4),
\\
& 
 u_\varepsilon(0)=0, \; 
  u_\varepsilon(4)=  1.
\end{aligned}
\right.
\end{align}
Denote by 
\begin{align*} 
\alpha_\varepsilon  :=\inf \{x \in (0,4) \, : \, u_\varepsilon(x)>0\}
\end{align*}
the free boundary point of the problem. This is a unique point since, by the minimum principle,~$u_\varepsilon$ is positive in~$(\alpha_\ep, 4)$ and zero on~$[0,\alpha_\ep]$. After a direct computation, the solution can be written as 
\begin{equation}
u_\varepsilon(x)
= 
\int_{\alpha_\ep}^{x \wedge \alpha_\ep} (t - \lambda_\ep) a^{-1}(\tfrac{t}{\varepsilon})\, dt ,
\quad
\end{equation}
where~$\lambda_\ep$ is chosen as
\begin{align}  \label{e.lambdaep}
\lambda_\ep 
= 
\Big( \int_{\alpha_\ep}^{4} t a^{-1}(\tfrac{t }{\varepsilon})\, dt - 1 \Big)
\Big( \int_{\alpha_\ep}^{4} a^{-1}(\tfrac{t}{\varepsilon})\, dt \Big)^{-1} 
\end{align}
to guarantee the boundary condition~$u_\varepsilon(4)=1$. Since~$a$ takes values less than~$2$, we see that~$\lambda_\ep$ is positive.  Employing then the fact that~$u_\varepsilon$ is~$C^1$ in the interior~$(0,2)$, we also obtain that~$u_\varepsilon^\prime(\alpha_\varepsilon)=0$, which gives us~$\alpha_\ep = \lambda_\ep$ and thus~\eqref{e.lambdaep} gives us the equation to solve~$\alpha_\ep$. 

\smallskip

Next, we let~$\bar a_\ep$ to be the coarsened coefficient and~$\phi_\ep$ the corrector: 
\begin{align*} 
\bar a_\ep 
= 
\Big( \frac14 \int_0^4 a^{-1}(\tfrac{x}{\varepsilon}) \, dx  \Big)^{-1} 
\quad \mbox{and} \quad 
\phi_\ep(x) = \bar a_\ep \int_0^x a^{-1}(\tfrac{t}{\varepsilon}) \, dt - x
\, .
\end{align*}
We assume the bound 
\begin{align}  \label{e.oned.homogass}
| \bar a_\ep  - \bar a | + \| \phi_\ep \|_{L^2(0,4)} \leq \mathcal{E}(\ep)
\end{align}
and~$\mathcal{E}(\ep) \to 0$ as~$\ep \to 0$.

Now let
\begin{equation}
\bar{a}:=\mathbb{E} \left( \int_0^1 a^{-1}(t)dt \right)^{-1},
\end{equation}
be the homogenized matrix, and~$\bar{u}$ be the solution of the obstacle problem 
solves 
\begin{align} \notag 
\left\{
\begin{aligned}
 & 
(\bar a \, \bar u^\prime)^\prime =\chi_{\{\bar u>0\}}
 \quad  \mbox{in } (0,4),
\\
& 
\bar u(0)=0, \; 
\bar u(4)=  1.
\end{aligned}
\right.
\end{align}
As above, a straightforward computation yields that 
\begin{align*}
\bar u (x)
= 
\bar a^{-1} \int_{\bar \alpha}^{x \wedge \bar \alpha} (t - \bar \alpha) \, dt ,
\quad
\mbox{with}
\quad
\bar \alpha
=  4 - \sqrt{2 \bar a }
\end{align*}
We  will then verify that 
\begin{enumerate}
  \item~$\alpha_\varepsilon \rightarrow \bar{\alpha}$ a.s. as~$\varepsilon \rightarrow 0$, 
  \item~$u_\varepsilon \rightarrow \bar{u}$ a.s. as~$\varepsilon \rightarrow 0$.
\end{enumerate}
Recall that~$\alpha_\varepsilon$ and~$\bar{\alpha}$ are determined by the following equations
\begin{align}
\notag
\int_0^4 (x-\alpha_\varepsilon)_+a^{-1}(\tfrac{x}{\varepsilon})dx=1
\quad 
\textrm{and}
\quad 
\int_0^4 (x-\bar{\alpha})_+\bar{a}^{-1}dx=1.
\end{align}
On the one hand, we have 
\begin{align}
\notag
\int_0^4 (x-\bar{\alpha})_+\bar{a}^{-1}dx = \frac12 \bar a^{-1}  (4-\bar \alpha)^2 . 
\end{align}
By integration by parts, we obtain
\begin{align*} 
\int_0^4 
(x-{\alpha_\varepsilon})_+  a^{-1}(\tfrac{x}{\varepsilon}) \, dx 
= \frac12 \bar a_\ep^{-1}  (4-{\alpha_\varepsilon})^2
-
\bar a_\ep^{-1}  
\int_{\alpha_\varepsilon}^4  \phi(x)  \, dx
.
\end{align*} 
Therefore, it follows, after some manipulations, that
\begin{align*} 
 \bar \alpha - \alpha_\varepsilon 
=
\frac{1}{8 - \bar \alpha - \alpha_\varepsilon} \biggl( 
\bar a_\ep (\bar a^{-1} - \bar a_\ep^{-1})
+  
2 \int_{\alpha_\varepsilon}^4  \phi_\ep(x)  \, dx
\biggr)
\end{align*}
The assumption~\eqref{e.oned.homogass} then implies that 
\begin{align*} 
| \bar \alpha - \alpha_\varepsilon | \leq C  \mathcal{E}(\ep) .
\end{align*}


Therefore~$\alpha_\varepsilon \rightarrow \bar{\alpha}$.
Using the explicit formulas for~$u_\varepsilon$ and for~$\bar{u}$, together with the convergence 
$\alpha_\varepsilon \rightarrow \bar{\alpha}$, we obtain the  convergence 
$ u_\varepsilon \rightarrow \bar{u}$. Furthermore, the calculations show that in the periodic case 
\begin{align*}
 \| u_\varepsilon-\bar{u} \|_{L^\infty} \leq C \varepsilon \textrm{ and }
 | \alpha_\varepsilon-\bar{\alpha} | \leq C \varepsilon.
 \end{align*}


\smallskip

\subsection{The homogenization of the free boundary  does not hold with a fixed obstacle}

We next show an example demonstrating the instability of the free boundary in homogenization with a fixed obstacle. Indeed, one cannot expect homogenization to hold even if the coefficients are smooth.

\smallskip

To this end, let~$\varphi~$Be a given concave function in the interval~$(-1,1)$, called an obstacle, we study the solution to the one-dimensional obstacle problem
\begin{align}\label{Ahom}
u \geq \varphi,~ (\bar{a} u^\prime)^\prime \leq 0, \textrm{ and } ~ (\bar{a} u^\prime)^\prime = 0 \textrm{ in }
\{ u > \varphi\}.
\end{align}
and the solutions~$u_\ep$ to corresponding heterogenous obstacle problem 
\begin{align}\label{Ahet}
u_\ep \geq \varphi,~ ({a^\ep} u_\ep^\prime)^\prime \leq 0, \textrm{ and } ~ ({a_\ep} u_\ep^\prime)^\prime = 0 \textrm{ in }
\{ u_\ep > \varphi\},
\end{align}
with boundary condition~$u^\ep(x)=u(x)$ when~$x=\pm 1$.

In order to simplify our calculations, we can take~$\varphi =\frac{1}{2} -{x^2}$, and the boundary conditions~$u^\ep(\pm 1)=u(\pm 1)=0$.

\begin{figure}
\begin{center}
\includegraphics[scale=0.40]{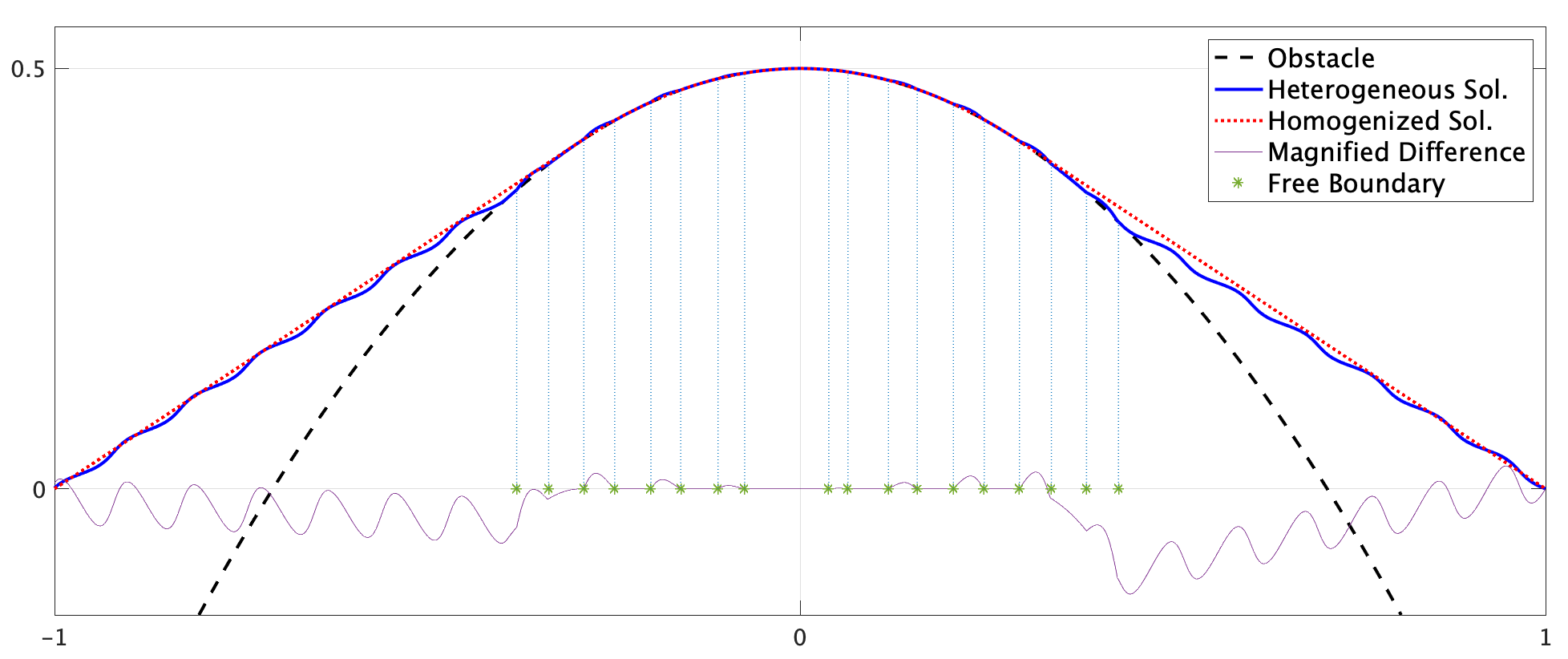}
\caption{Given obstacle $\varphi = \frac{1}{2} - x^2$, there are plotted both heterogeneous and homogenized solutions, together with $\varphi$ and the free boundary. Also the difference, magnified by a factor of $6$, between heterogeneous and homogenized solutions is plotted.
}
\end{center}
\label{f.oned.example}
\end{figure}

By~\eqref{Ahom},~$u$ is a  linear function in every connected component of  the set~$\{ u > \varphi\}$,   furthermore, 
if~$\bar{x}$ is the free boundary point closest to the boundary point~$-1$, then by straightforward calculations,~$u(x)= \varphi^\prime(\bar{x}) (x-\bar{x}) + \varphi{(\bar{x})}$ in the interval~$(-1, \bar{x})$ (The first order Taylor polynomial of~$\varphi$ at~$\bar{x}$).
Letting~$u(-1)=0$, we obtain~$\bar{x}=-1+\frac{\sqrt{2}}{2}$.\\
Now let us see what similar calculations give for~$u_\ep$. Denote by~$x_\ep$  the free boundary point closest to the boundary point~$-1$, 
then 
\begin{align*}
(a(\tfrac{\cdot}{\ep})u^\prime_\ep)^\prime=0 \textrm{ in the interval } (-1,x_\ep),
\end{align*}
and at the free boundary point~$x_\ep$  we get 
\begin{align*}
u_\ep( x_\ep)=\varphi (x_\ep), ~ u^\prime_\ep( x_\ep)=\varphi^\prime (x_\ep),
\end{align*}
Hence
\begin{align*}
u_\ep(x )=a(\tfrac{x_\ep}{\ep}) \varphi^\prime(x_\ep) \int_{x_\ep}^x a^{-1}(\tfrac{t}{\ep})dt +\varphi( x_\ep),  \textrm{ in the interval } (-1,x_\ep). 
\end{align*}
Letting~$u_\ep(-1)=0$, we get
\begin{align} \label{fbpep}
 x_\ep a(\tfrac{x_\ep}{\ep}) \int^{x_\ep}_{-1} a^{-1}(\tfrac{t}{\ep})dt +\frac{1}{2}- x_\ep^2=0
\end{align}

Assuming that~$x_\ep \to \bar{x}$ as~$\ep \to 0$, we get a contradiction in \eqref{fbpep} due to the oscillation of~$a(\tfrac{x_\ep}{\ep})$. Hence no homogenization of free boundary result can be expected given a fixed obstacle~$\varphi$.
This happens, since we have no control over~$\nabla \cdot (\a^\ep \nabla \varphi)$ as~$\ep \to 0+$.
This example fully justifies the normalization of the right hand side in the  context of homogenization of the obstacle problem.


\small
\bibliographystyle{abbrv}
\bibliography{obstacle-AK}

\end{document}